\documentclass{article}

\usepackage{amsmath}
\usepackage{amssymb}
\usepackage{amsthm}
\usepackage{hyperref}

\newtheorem{theorem}{Theorem}[section]
\newtheorem{lemma}[theorem]{Lemma}
\newtheorem{proposition}[theorem]{Proposition}
\newtheorem{definition}[theorem]{Definition}
\newtheorem{corollary}[theorem]{Corollary}
\newtheorem{remark}[theorem]{Remark}

\usepackage{tikz}
\usetikzlibrary{cd}

\DeclareMathOperator{\Sp}{Sp}
\DeclareMathOperator{\Mod}{Mod}

\DeclareMathOperator{\Orb}{Orb}
\DeclareMathOperator{\Pic}{Pic}
\DeclareMathOperator{\Hom}{Hom}

\DeclareMathOperator{\map}{map}
\DeclareMathOperator{\Map}{Map}

\DeclareMathOperator{\Fun}{Fun}
\DeclareMathOperator{\End}{End}

\DeclareMathOperator{\RO}{RO}
\DeclareMathOperator{\Ind}{Ind}
\DeclareMathOperator{\StMod}{StMod}
\DeclareMathOperator{\cPic}{\mathcal{P}ic}

\newcommand{\bS}{\mathbb{S}}
\newcommand{\bZ}{\mathbb{Z}}
\newcommand{\Z}{\mathbb{Z}}
\newcommand{\bR}{\mathbb{R}}

\newcommand{\F}{\mathbb{F}}
\newcommand{\cF}{\mathcal{F}}
\newcommand{\cA}{\mathcal{A}}
\newcommand{\cS}{\mathcal{S}}
\newcommand{\cB}{\mathcal{B}}
\newcommand{\cC}{\mathcal{C}}
\newcommand{\cD}{\mathcal{D}}

\newcommand{\unit}{\mathbb{I}}
\newcommand{\gen}{\mathrm{gen}}

\newcommand{\fin}{\mathrm{fin}}
\newcommand{\colim}{\operatornamewithlimits{colim}}
\newcommand{\xto}[1]{\xrightarrow{#1}}
\newcommand{\underlinemap}{\underline{\map}}

\title{The Picard group in equivariant homotopy theory via stable module categories}
\author{Achim Krause}

\begin{document}
\maketitle

\begin{abstract}
We develop a mechanism of ``isotropy separation for compact objects'' that explicitly describes an invertible $G$-spectrum through its collection of geometric fixed points and gluing data located in certain variants of the stable module category. As an application, we carry out a complete analysis of invertible $G$-spectra in the case $G=A_5$. A further application is given by showing that the Picard groups of $\Sp^G$ and a category of derived Mackey functors agree.
\end{abstract}

\section{Introduction}

The $\infty$-category $\Sp^G$ of genuine $G$-spectra for a finite group $G$, obtained from $G$-spaces by stabilisation with respect to all representation spheres, has a very interesting Picard group. By construction, representation spheres $S^V$ give invertible objects in $\Sp^G$, but the corresponding homomorphism $\RO(G)\to \Pic(\Sp^G)$ is generally neither injective nor surjective. For this reason, invertible $G$-spectra are commonly called (stable) homotopy representations of $G$. They have been studied quite a lot, for example in \cite{tomdieckpetrie}, \cite{bauer}, \cite{fausklewismay}.

One can check that invertible objects of $\Sp^G$ are precisely given by compact objects $X$ all of whose geometric fixed points $\Phi^H(X) \in \Sp$ are equivalent to shifts $S^{d_H}$ of the sphere spectrum. By recording the dimension of those spheres into a \emph{dimension function} $H\mapsto d_H$, one obtains a homomorphism
\begin{equation}
\label{eq:dimensionhom}
\Pic(\Sp^G) \to \prod_{\Orb(G)} \Z
\end{equation}
where the product is over isomorphism classes of orbits of $G$ (or equivalently, conjugacy classes of subgroups of $G$).
This dimension homomorphism is studied by tom Dieck-Petrie \cite{tomdieckpetrie} to great effect: They identify the kernel with $\Pic(A(G))$, the Picard group of the Burnside ring of $G$, which is finite, and they determine the image rationally, by explicitly describing a certain subgroup of $\prod_{\Orb(G)} \Z$ and proving that the dimension homomorphism factors over a rational isomorphism onto that subgroup. The subgroup is precisely cut out by relations arising from subquotients $C_p\times C_p$ of $G$, due to Borel \cite{borel}. In particular, tom Dieck and Petrie prove that the dimension homomorphism is a rational isomorphism if $G$ admits no subquotients of the form $C_p\times C_p$.

Integrally, the image behaves in an interesting way. For example, the result by tom Dieck and Petrie shows abstractly that for an odd integer $n$, there is some $d\neq 0$ and some invertible $D_{2n}$-spectrum $X\in \Pic(\Sp^{D_{2n}})$ with underlying object $S^d$, but $\Phi^H X \simeq S^0$ for each nontrivial subgroup $H\subseteq G$. One can check that the possible $d$ are precisely the multiples of $4$. More generally, a complete description of the group of dimension functions realizable by invertible $G$-spectra is obtained by Bauer \cite{bauer} in terms of the representation theory of $G$.

In this paper, we exhibit an interesting connection between the dimension functions of invertible objects of $\Sp^G$, and the \emph{stable module category} of the group $G$. This is roughly obtained by localising finitely generated $G$-modules by finitely generated projective ones, and is closely related to Tate cohomology. We will see that the above example of invertible objects over $D_{2n}$ is closely connected to the $4$-periodicity of the Tate cohomology of $D_{2n}$, and more generally we will exhibit a mechanism by which $\Pic(\Sp^G)$ is controlled by the Picard groups of stable module categories associated to the Weyl groups of subgroups of $G$.

Invertible objects of stable module categories are known in modular representation theory as \emph{endotrivial modules}. Much work has been done to understand these in the abelian and $p$-group case, see \cite{dade1}, \cite{dade2},  \cite{alperin}, \cite{carlsonthevenaz1}, \cite{carlsonthevenaz2}, and there are strong results on how the Picard group of the stable module category of a group $G$ is controlled by its $p$-subgroups \cite{grodal}. 

Our key tool will be a variant of isotropy separation interacting well with compact objects and monoidal structures. Roughly, a compact $G$-spectrum will be described by its various geometric fixed points together with gluing data in certain stable module categories. This approach gives a very satisfying categorical perspective on several classical unstable arguments based on explicit attachment of cells. In particular, we obtain an inductive description of the Picard $\infty$-groupoid $\cPic(\Sp^G)$. Here $\Sp^G_\omega/\cF$ denotes a certain localisation of $\Sp^G_\omega$ obtained by killing isotropy from a family of orbits $\cF\subseteq \Orb(G)$.

\begin{theorem}
\label{thm:pullbackintro}
For a family $\cF$ of orbits and $G/K\not\in \cF$ with $G/K'\in \cF$ for all $K'\subsetneq K$, there is a $1$-cartesian diagram of Picard $\infty$-groupoids
\begin{equation}
\label{eq:pullbackintro}
\begin{tikzcd}
\cPic(\Sp_\omega^G/\cF) \rar\dar{H\Phi^K} & \cPic(\Sp_\omega^G/\cF\cup \{G/K\})\dar{H\overline{\Phi}^K}\\
\cPic(\Fun(BW_G(K),\Mod_\omega(\Z))) \rar & \cPic(\StMod(\Z W_G(K))).
\end{tikzcd}
\end{equation}
In particular, this diagram gives a pullback diagram on homotopy categories.
\end{theorem}
Inductively, this describes a Picard object $X$ of $\Sp^G$ through the homologies $H_*(\Phi^K X)$ together with gluing isomorphisms in $\StMod(\Z W_G(K))$. This is not only quite computable, which we illustrate by computing the dimension functions in the case $G=A_5$ in Section \ref{sec:applications}, it also has interesting theoretical applications. Indeed, the diagram \eqref{eq:pullbackintro} admits a version with coefficients, computing $\cPic(\Mod(R_G))$ instead of $\cPic(\Sp^G)$ for $R$ any ordinary connective commutative ring spectrum, and $R_G\in \Sp^G$ the equivariant spectrum obtained by restriction along $G\to e$ (usually called inflation). Diagram \eqref{eq:pullbackintro} then shows that the Picard group of $R_G$ only depends on $\pi_0(R)$. For example, we get:
\begin{theorem}
\label{thm:linearisationintro}
$\Pic(\Sp^G)$ agrees with the Picard group of $\Mod_{\Sp^G}(H\Z_G)$.
\end{theorem}
Here $H\Z_G$, the inflation of $H\Z\in \Sp$ along $G\to e$, is the $\Z$-linearisation of the genuine sphere. It is a genuine $G$-spectrum with all geometric fixed points given by $H\Z$ (not to be confused with a Mackey functor Eilenberg-Mac Lane spectrum). The category $\Mod_{\Sp^G}(H\Z_G)$ is equivalent to the category of coherent Mackey functors with values in $\Mod_{\Sp}(H\Z) \simeq \cD(\Z)$, see \cite[Corollary 4.11]{patchkoriasanderswimmer}.
In that sense, Theorem \ref{thm:linearisationintro} should be regarded as a more algebraic perspective on the image of the dimension homomorphism \ref{eq:dimensionhom}, even though Mackey functors with values in $\cD(\Z)$ still involve complicated coherences.

Similar results have been obtained by Grodal-Smith \cite{grodalsmithoberwolfach} by different methods.

\subsection{Acknowledgements}

I am grateful to John Greenlees, Jesper Grodal, Drew Heard, Tyler Lawson, Akhil Mathew, Thomas Nikolaus, and Kaif Tan for very helpful discussions related to this project.

The author was funded by the Deutsche Forschungsgemeinschaft (DFG, German Research Foundation) under Germany’s Excellence Strategy EXC 2044 390685587, Mathematics M\"unster: Dynamics-Geometry-Structure, as well as under Project-ID 427320536 - SFB 1442.

\section{Preliminaries}

In this note, we work with the $\infty$-category $\cS^G$ (resp. $\cS^G_*$) of $G$-spaces (resp. pointed $G$-spaces), and the $\infty$-category $\Sp^G$ of genuine $G$-spectra. The perspective we take is rather axiomatic, indeed, we will use just the following facts. These are well-known and can be proven in each of the common models for genuine equivariant homotopy theory, see e.g. \cite{mandellmay}, \cite{greenleesmay}, \cite{schwede}, \cite{barwick}, \cite{barwicketal} or the appendix of \cite{hillhopkinsravenel}.

\begin{enumerate}
\item $\Sp^G$ is a stable, symmetric-monoidal $\infty$-category. There is a symmetric-monoidal functor $\Sigma^\infty: \cS^G_* \to \Sp^G$ from the $\infty$-category of pointed $G$-spaces, endowed with the smash product. $\Sigma^\infty$ admits a right adjoint $\Omega^\infty$ which commutes with filtered colimits and is faithful.
\item As a consequence of (1), $\Sigma^\infty$ preserves compact objects, and the orbits $\Sigma^\infty_+ G/H$ form a system of compact generators.
\item There are symmetric-monoidal functors $\Phi^H: \Sp^G\to \Sp$, which preserve colimits and satisfy $\Phi^H \Sigma^\infty X \simeq \Sigma^\infty X^H$.
\item For $X\in \Sp^G$ with $\Phi^H X \simeq 0$ for all $H\subsetneq K$, we have a natural equivalence $\Phi^K X \simeq X^K := \map_{\Sp^G}(\Sigma^\infty_+ G/K, X)$. In particular, if all $\Phi^K X \simeq 0$, then all $X^K\simeq 0$, so $X\simeq 0$. Thus, the $\Phi^K$ together detect equivalences.
\item The $\Sigma^\infty_+ G/H$ are all dualizable in $\Sp^G$. Since these are compact generators, and dualizable objects are closed under extensions and retracts, it follows that all compact objects in $\Sp^G$ are dualizable. Conversely, all dualizable objects are compact, since the unit $\bS_G = \Sigma^\infty_+ G/G$ is.
\end{enumerate}

For a presentably symmetric-monoidal $\infty$-category $\cC$, recall that we can form the Picard $\infty$-groupoid $\cPic(\cC)$. This is always small: If the unit of $\cC$ is $\kappa$-compact, it follows that all invertible objects are $\kappa$-compact, and since $\cC_\kappa$ is essentially small, it follows that $\cPic(\cC)$ is essentially small, too.
Thus, $\cPic(\cC)$ can be considered as a space (in fact, the infinite loop space of a spectrum $\mathfrak{pic}(\cC)$, due to the symmetric-monoidal structure), with $\pi_0(\cPic(\cC)) = \Pic(\cC)$, the usual Picard group of $\cC$, consisting of isomorphism classes of invertible objects. $\cPic(\cC)$ has higher homotopy groups as well: $\pi_1(\cPic(\cC))$ is given by homotopy classes of automorphisms of the unit of $\cC$, and for $n\geq 2$ $\pi_n(\cPic(\cC))$ agrees with $\pi_{n-1}\End_{\cC}(\unit)$.
The advantage of $\cPic(\cC)$ over the mere Picard group $\Pic(\cC)$ is that $\cPic$ commutes with limits of symmetric-monoidal $\infty$-categories, see \cite[Prop. 2.2.3]{mathewstojanoska}. This enables us to apply descent methods to the study of the Picard group.

Since $\cPic$ is functorial in symmetric-monoidal functors, geometric fixed points $\Phi^H: \Sp^G\to \Sp$ induce maps $\cPic(\Sp^G)\to \cPic(\Sp)$. This is just expressing the fact that for $X$ an invertible $G$-spectrum, $\Phi^H X$ is an invertible spectrum, i.e. a shift of the sphere spectrum.
A converse of this fact holds as well: Given a \emph{compact} $G$-spectrum $X\in \Sp^G_\omega$, then $X$ is invertible if and only if all its geometric fixed points $\Phi^H X$ are.
To see this, note that $X$ is dualizable, so we have a dual $DX$ with canonical evaluation map
\[
DX \otimes X\to \bS_G.
\]
The assertion that $X$ is invertible is equivalent to this being an equivalence. But since $\Phi^H$ is symmetric-monoidal, thus preserves duals, and all $\Phi^H$ together detect equivalences, this is equivalent to all $\Phi^H X$ being invertible.
Note that without the a priori assumption of compactness of $X$, knowing that the $\Phi^H X$ are invertible is not enough. For example, one can find $C_2$-spectra $X$ and $Y$ with $\Phi^{C_2} X \simeq 0$, $\Phi^e X \simeq \bS$, and $\Phi^e Y\simeq 0$, $\Phi^{C_2}Y\simeq \bS$, so their sum $X\oplus Y$ has all geometric fixed points given by invertible spectra. But $X$ and $Y$ as above can never be compact, and so in particular $X\oplus Y$ cannot be invertible.

The goal of this paper is to make this relationship between invertible objects of $\Sp^G_\omega$ and their geometric fixed points tighter, and to explicitly describe the data required to reconstruct $X$ out of the various $\Phi^H X$.

\section{Isotropy separation for compact objects}

\begin{definition}
We call $\cF\subseteq \Orb(G)$ a family of orbits of $G$, if whenever $G/H \in \cF$ and $H'\subseteq H$, we have $G/H'\in \cF$.
\end{definition}

For such a family, there are cofiber sequences in $\Sp^G$, functorial in $X$, of the form
\[
X_{\cF} \to X \to X_{\cF^c}
\]
where $\Phi^H X_{\cF^c} \simeq 0$ whenever $G/H\in \cF$, and $\Phi^H X_{\cF} \simeq 0$ whenever $G/H\not\in \cF$. This cofiber sequence is part of a \emph{recollement}, which allows one to describe the category of genuine $G$-spectra in terms of certain categories of $G$-spectra with isotropy supported at $\cF$ or away from $\cF$, arising as Bousfield localisations of $\Sp^G$. This is described in detail in \cite{quigleyshah}. Inductively, one can describe $\Sp^G$ in terms of a sequence of recollements which exhibits it as built from pieces associated to one isotropy subgroup at a time. For example, in the case of $G=C_p$, where there are only two subgroups, this description takes the form of a pullback diagram (cf. \cite[Theorem 6.24]{mathewnaumannnoel})
\begin{equation}
\label{eq:recollpullback}
\begin{tikzcd}
\Sp^{C_p} \rar\dar & \Fun(\Delta^1,\Sp)\dar{t}\\
\Fun(BC_p, \Sp) \rar& \Sp
\end{tikzcd}
\end{equation}
where the left vertical functor takes a $G$-spectrum $X$ to its underlying spectrum $\Phi^e X$ with the action by the Weyl group $W_{C_p}(e)=C_p$, and the top vertical functor takes $X$ to the arrow $\Phi^{C_p} X \to (\Phi^e X)^{tC_p}$. On the level of objects, this square precisely says that a $C_p$-spectrum $X$ can be recovered from this data, i.e. a spectrum $\Phi^{C_p} X$, a spectrum $\Phi^e X$ with $C_p$-action, and some map 
\begin{equation}
\label{eq:gluingmap}
\Phi^{C_p} X \to (\Phi^e X)^{tC_p}.
\end{equation}
On morphisms, this pullback square leads to the familiar Tate square on mapping spectra,
\begin{equation}
\label{eq:tatesquare}
\begin{tikzcd}
\map_{\Sp^{C_p}}(X,Y)\rar\dar & \underlinemap(X,Y)^{\Phi^{C_p}}\dar\\
\underlinemap(X,Y)^{hC_p} \rar & \underlinemap(X,Y)^{tC_p},
\end{tikzcd}
\end{equation}
valid for compact $X\in \Sp^G$. Here $\underlinemap$ denotes the internal Hom of $\Sp^{C_p}$. One should consider the upper right corner to measure maps on the part with isotropy $C_p$ (i.e. geometric $C_p$-fixed points), the lower left corner to measure maps on the part with trivial isotropy (i.e. geometric fixed points with respect to the trivial subgroup $e\subseteq C_p$), and the lower right corner to somehow control the gluing of these parts.

While very satisfying, this perspective is not well-suited to control $\Pic(\Sp^G)$: Of course, for $X$ to be invertible, it is necessary that $\Phi^eX$ and $\Phi^{C_p} X$ are invertible spectra, but in addition $X$ needs to be a compact object of $\Sp^G$. This translates to a condition on the map \eqref{eq:gluingmap} which seems hard to determine explicitly. This is closely related to the fact that the pieces $X_{\cF}$ and $X_{\cF^c}$ in the general isotropy separation sequence admit a description in terms of the classifying space $E\cF$ of the family $\cF$ of subgroups, which is basically never compact.

So we take a different route, and work with compact objects to begin with. In this picture, we will derive a pullback diagram different from \eqref{eq:recollpullback}, which looks much more like a direct categorification of \eqref{eq:tatesquare}. Specifically, in the special case of $G=C_p$, compact objects $X\in\Sp_{\omega}^{C_p}$ will correspond to a pair of compact objects $\Phi^e X\in \Fun(BC_p,\Sp)$ and $\Phi^{C_p} X\in \Sp$ as before, but now with an \emph{equivalence} between their images in a certain category which we view as a version of the stable module category with coefficients in the sphere. The corresponding pullback diagram on morphisms will directly recover the Tate square.

First recall the following construction:
\begin{definition}
For a stable $\infty$-category $\cB$ with stable full subcategory $\cA$, a map in $\cB$ is a \emph{mod $\cA$ equivalence}, if its fiber is equivalent to an object of $\cA$. The Verdier quotient $\cB/\cA$ is the Dwyer-Kan localization of $\cB$ with respect to all mod $\cA$ equivalences.
\end{definition}

\begin{lemma}
\label{lem:mapsinquotient}
Mapping spectra in $\cB/\cA$ admit the following two descriptions:
\begin{enumerate}
\item $\map_{\cB/\cA}(X,Y)\simeq \colim_{X'\xto{\sim} X} \map_{\cB}(X',Y)$, where the colimit ranges over (op of) the cofiltered diagram of all mod $\cA$ equivalences $X'\to X$ over $X$.
\item $\map_{\cB/\cA}(X,Y) \simeq \colim_{Y\xto{\sim} Y'} \map_{\cB}(X,Y')$, where the colimit ranges over the filtered diagram of all mod $\cA$ equivalences $Y\to Y'$ under $Y$.
\end{enumerate}
\end{lemma}
\begin{proof}
See \cite[Theorem I.3.3]{nikolausscholze}, which proves (2) as (ii). (1) follows analogously, or by replacing $\cB$ by its opposite category.
\end{proof}

\begin{corollary}
\label{cor:kernelquotient}
The maps in $\cB$ which become equivalences in $\cB/\cA$ are precisely those whose fiber is a retract of an object in $\cA$. In particular, if $\cA$ is closed under retracts in $\cB$, the maps that become equivalences in $\cB/\cA$ are precisely the equivalences mod $\cA$. 
\end{corollary}
\begin{proof}
Since the functor $\cB\to \cB/\cA$ is exact, a map goes to an equivalence precisely if its fiber is mapped to $0$. An object is zero precisely if its identity agrees with its zero endomorphism, by the description from Lemma \ref{lem:mapsinquotient} this happens to $X$ precisely if there is an $X'$ such that the zero map $X'\to X$ is a mod $\cA$ equivalence, i.e. if $X$ is a summand of an object in $\cA$.
\end{proof}

\begin{lemma}
\label{lem:quotientmonoidal}
Let $\cB$ be symmetric-monoidal, and let $\cA$ be a stable full subcategory. There exists an essentially unique symmetric monoidal structure on $\cB/\cA$ and the functor $\cB\to \cB/\cA$ precisely if the subcategory $\cA$ is an ideal, i.e. for any $X\in \cA$ and $Y\in \cB$, we have $X\otimes Y\in \cA$.
\end{lemma}
\begin{proof}
This is \cite[Theorem I.3.6]{nikolausscholze}.
\end{proof}

\begin{definition}
For a family $\cF$ of orbits of $G$, we let
\[
\Sp^G_{\omega}/\cF
\]
denote the Verdier quotient of $\Sp^G_{\omega}$ by the thick subcategory $\langle \Sigma^\infty_+ G/H\rangle_{G/H\in \cF}$ generated by the suspension spectra of orbits in $\cF$ (which are always compact) under cofibers, fibers and retracts. 
\end{definition}

Since $\cF$ is a family and the product of two $G$-orbits admits a double coset decomposition, one sees that Lemma \ref{lem:quotientmonoidal} makes $\Sp^G_\omega/\cF$ into a symmetric-monoidal localization of $\Sp^G_\omega$.

\begin{lemma}
\label{lem:geometricfixedpointscorepresented}
Let $\cF$ be a family of orbits, and let $G/K$ be an orbit not contained in $\cF$, but such that all $G/K'$ with $K'\subsetneq K$ are contained in $\cF$. For $X\in \Sp^G_\omega$ we then have
\[
\map_{\Sp^G_\omega/\cF}(\Sigma^\infty_+ G/K, X) \simeq \Phi^K X
\]
\end{lemma}
\begin{proof}
By Lemma \ref{lem:mapsinquotient}, we can describe the mapping spectrum in question as a filtered colimit
\[
\colim_{X \xto{\sim} X'} \map_{\Sp^G_\omega}(\Sigma^\infty_+ G/H, X').
\]
In the larger category $\Sp^G$, which contains $\Sp^G_\omega$ as a full subcategory, $\Sigma^\infty_+ G/K$ is compact and we can thus write the above diagram as
\[
\map_{\Sp^G}(\Sigma^\infty_+ G/K, \colim_{X\xto{\sim} X'} X') \simeq \big(\colim_{X\xto{\sim} X'} X'\big)^K,
\]
where the colimit is taken over all maps $X\to X'$ under $X$ with fiber in $\langle\Sigma^\infty_+ G/H\rangle_{G/H\in \cF}$.

The map $X\to \colim_{X\xto{\sim} X'} X'$ induces an isomorphism on geometric fixed points for all subgroups not contained in $\cF$, because objects of the thick subcategory $\langle \Sigma^\infty_+ G/H\rangle_{H\in \cF}$ have trivial geometric fixed points for groups not in $\cF$.
Furthermore, the $H$-fixed points $(\colim_{X\xto{\sim} X'} X')^H \simeq \colim_{X\xto{\sim} X'} (X')^H$ for every $H\in \cF$ vanish.
We check this on homotopy groups: Given an element $\alpha\in \pi_n(\colim_{X\xto{\sim} X'} (X')^H)$, it is represented by a map $S^n\otimes \Sigma^\infty_+ (G/H) \to X'$ for some mod $\langle G/H\rangle_{H\in \cF}$ equivalence $X\to X'$. But then the cofiber
\[
S^n\otimes \Sigma^\infty_+ G/H \to X' \to X''
\]
also gives an object $X\to X''$ in the filtered colimit diagram, and $\alpha$ is evidently zero in $\pi_n((X'')^H)$.

It follows immediately from the usual definition of geometric fixed points that a $G$-spectrum $Z$ whose fixed points vanish for all proper subgroups $K'\subsetneq K$ has $Z^K \simeq \Phi^K Z$. In our situation, this implies
\[
\map_{\Sp^G}(\Sigma^\infty_+ G/K, \colim_{X\xto{\sim} X'} X') \simeq \Phi^K \colim_{X\xto{\sim} X'} X' \simeq \Phi^K X,
\]
as desired.
\end{proof}

\begin{lemma}
\label{lem:geomfixedpointsiso}
The functor $\Phi^K: \Sp^G_\omega/\cF \to \Sp_\omega$ admits a refinement $\Sp^G_\omega/\cF \to \Fun(BW_G(K), \Sp_\omega)$, and this induces equivalences
\[
\map_{\Sp^G_\omega/\cF}(X,Y) \to \map_{\Fun(BW_G(K), \Sp_\omega)}(\Phi^K X, \Phi^K Y)
\]
whenever $X$ or $Y$ is contained in the thick subcategory $\langle \Sigma^\infty_+ G/K\rangle$ generated by $\Sigma^\infty_+ G/K$ in $\Sp^G_\omega/\cF$. 
\end{lemma}
\begin{proof}
$G/K$, considered as a $G$-space, admits a $W_G(K)$-action, i.e. a refinement to $\Fun(BW_G(K), \cS^G)$, where $\cS^G$ denotes the $\infty$-category of $G$-spaces. The desired refinement of $\Phi^K$ is then obtained by functoriality. For the claim about mapping spectra, we first observe that this is true for $X = \Sigma^\infty_+ G/K$. Indeed, $\Phi^K \Sigma^\infty_+ G/K = \Sigma^\infty_+ (G/K)^K = \Sigma^\infty_+ W_G(K)$ is free as a spectrum with $W_G(K)$-action, and thus
\[
\map_{\Sp^G_\omega/\cF}(\Sigma^\infty_+ G/K,Y) \to \map_{\Fun(BW_G(K), \Sp_\omega)}(\Phi^K \Sigma^\infty_+ G/K, \Phi^K Y) \simeq \Phi^K Y
\]
is seen to be the equivalence from Lemma \ref{lem:geometricfixedpointscorepresented}. Now note that the collection of all $X$ for which this map is an equivalence is certainly a thick subcategory since both sides are exact functors in $X$, so since it contains $\Sigma^\infty_+ G/K$, it contains all of $\langle \Sigma^\infty_+ G/K\rangle$.

For the corresponding statement where $Y$ is in the thick subcategory generated by $\Sigma^\infty_+ G/K$, we use duality: Every object of $\Sp^G_\omega$ is dualizable by equivariant Spanier-Whitehead duality, and since the functor to $\Sp^G_\omega/\cF$ is essentially surjective and symmetric-monoidal, this holds true in $\Sp^G_\omega/\cF$ as well. Also, $\Phi^K$ is symmetric-monoidal, so in particular preserves duals. It follows that we can replace $\map(X,Y)$ with $\map(DY,DX)$ on both sides, and since $D(\Sigma^\infty_+ G/K) \simeq \Sigma^\infty_+ G/K$, this reduces us to the previously established statement.
\end{proof}

We now analyze a relative situation, where we consider two families $\cF$, $\cF\cup \{G/K\}$ which only differ by one orbit. More precisely, for a family of orbits $\cF$ and $G/K\not\in \cF$ with all $G/K'$ for $K'\subsetneq K$ contained in $\cF$, $\cF\cup \{G/K\}$ is also a family of orbits. Observe that we then can consider $\Sp^G_\omega/\cF\cup\{G/K\}$ a localisation of $\Sp^G_\omega/\cF$, namely as the Verdier quotient by the thick subcategory $\langle \Sigma^\infty_+ G/K\rangle$ generated by $\Sigma^\infty_+ G/K$.

We are now in the abstract situation of the following lemma:

\begin{lemma}
\label{lem:pullbackquotient}
Let $\cA\subseteq \cB$ be a stable full subcategory of the stable $\infty$-category $\cB$, and let $F: \cB \to \cB'$ be an exact functor. Assume that 
\begin{enumerate}
\item $F$ induces an equivalence $\map_{\cB}(X,Y) \to \map_{\cB'}(FX,FY)$ whenever $X\in \cA$ or $Y\in \cA$. In particular, $F|_{\cA}$ is fully faithful.
\item $\cA$ is closed under retracts in $\cB$, and the essential image $F(\cA)$ is closed under retracts in $\cB'$.
\end{enumerate}
Then we have a pullback diagram of Verdier quotients
\[
\begin{tikzcd}
\cB \rar\dar{F} & \cB/\cA\dar\\
\cB'\rar & \cB'/F(\cA)
\end{tikzcd}
\]
\end{lemma}
\begin{proof}
We first check that the functor from $\cB$ into the pullback is fully faithful. This amounts to checking that the diagram
\[
\begin{tikzcd}
\map_{\cB}(X,Y) \rar\dar & \colim_{X' \xto{\sim} X}\map_{\cB}(X',Y)\dar\\
\map_{\cB'}(FX, FY) \rar & \colim_{X' \xto{\sim} FX} \map_{\cB'}(X', FY)
\end{tikzcd}
\]
is a pullback diagram. We check that the map between the horizontal fibers is an equivalence. The fiber of the upper horizontal map admits a description as
\[
\colim_{(X\to A) \in \cA_{X/}} \map_{\cB}(A, Y)
\]
where we have exploited that an equivalence mod $\cA$ $X'\to X$ is the fiber of its cofiber $X\to A$, $A\in \cA$. Similarly, the fiber of the lower horizontal map admits a description as
\[
\colim_{(FX \to A)\in F\cA_{FX/}} \map_{\cB'}(A, FY).
\]
By the first assumption, these diagrams agree.

We now show the functor from $\cB$ to the pullback is essentially surjective. We thus have to show that, given an object $b\in \cB$ (representing some object in $\cB/\cA$) and an object $b'\in \cB'$, together with an equivalence between $b'$ and $Fb$ in $\cB'/\cA$, we can find a new object $\widetilde{b'}$ in $\cB$, lifting $b'$ under $F$, and an equivalence between $\widetilde{b'}$ and $b$ in $\cB/\cA$ lifting the previously given equivalence. Since equivalences in $\cB/\cA$ are precisely zigzags of maps with fiber in $\cA$ by Corollary \ref{cor:kernelquotient}, we are reduced to check the following two facts:

\begin{enumerate}
\item Given an object $b'\in \cB'$, an object $b\in \cB$, and a map $Fb \to b'$ in $\cB'$ with fiber in $F\cA$, there is a map $b\to \widetilde{b'}$ in $\cB$, with fiber in $\cA$, lifting the map $Fb\to b'$.
\item Given an object $b'\in \cB'$, an object $b\in \cB$, and a map $b'\to Fb$ in $\cB'$ with fiber in $F\cA$, there is a map $\widetilde{b'}\to b$ in $\cB$, with fiber in $\cA$, lifting the map $Fb\to b'$.
\end{enumerate}

In the first situation, we let $a\in F\cA$ be the fiber of $Fb\to b'$. By assumption, there is an essentially unique $\widetilde{a}\to b$, with $\widetilde{a}\in \cA$, lifting the morphism $a\to Fb$, and the cofiber of $\widetilde{a}\to b$ is the desired lift. In the second situation, we instead let $a\in F\cA$ be the cofiber of $b'\to Fb$. Again we have an essentially unique $b\to \widetilde{a}$ lifting the morphism $Fb\to a$, and its fiber is the desired lift.
\end{proof}

\begin{theorem}
\label{thm:pullback}
Let $\cF$ be a family of orbits of $G$ and $G/K\not\in \cF$ with $G/K'\in \cF$ for all $K'\subsetneq \cF$. Then there is a pullback diagram of symmetric-monoidal stable $\infty$-categories
\[
\begin{tikzcd}
\Sp^G_\omega/\cF \rar\dar{\Phi^K} & \Sp^G_\omega/\cF\cup\{G/K\}\dar{\overline{\Phi}^K}\\
\Fun(BW_G(K), \Sp_\omega) \rar & \Fun(BW_G(K),\Sp_\omega)/\langle\bS[W_G(K)]\rangle
\end{tikzcd}
\]
\end{theorem}
\begin{proof}
We apply Lemma \ref{lem:pullbackquotient}. By Lemma \ref{lem:geomfixedpointsiso}, the first condition is satisfied. The second condition requires some justification, as $\langle \Sigma^\infty_+ G/K\rangle$ upstairs and $\langle \bS[W_G(K)]\rangle$ downstairs are both by definition closed under retracts in the respective ambient category, but it is not a priori clear that the essential image of $\Phi^K|_{\langle \Sigma^\infty_+ G/K\rangle}$ agrees with $\langle \bS[W_G(K)]\rangle$ (this is closely related to the fact that $\Sp^G_\omega/\cF$ is not necessarily idempotent complete, so $\langle \Sigma^\infty_+ G/K\rangle$ is not necessarily idempotent complete).

However, there is also a functor $\Fun(BW_G(K), \Sp) \to \Sp^G$, given by sending a spectrum $X$ with $W_G(K)$-action to $X\otimes_{W_G(K)} \Sigma^\infty_+ G/K$, and it takes $\langle \bS[W_G(K)]\rangle$ to the thick subcategory $\langle\Sigma^\infty_+ G/K\rangle$ in $\Sp^G$. After a further composite to $\Sp^G_\omega/\cF$, this splits the functor $\Phi^K: \langle \Sigma^\infty_+ G/K\rangle\to \langle \bS[W_G(K)]\rangle$ above and shows that it is essentially surjective, which proves the theorem.
\end{proof}

We can also give a version with coefficients. To state this, let $X\in \Sp$ be an ordinary spectrum. We obtain a genuine spectrum $X_G$ in genuine $G$-spectra by restricting $X\in \Sp=\Sp^e$ along the map $G\to e$ (this is commonly called inflation rather than restriction). This construction $\Sp\to \Sp^G$ commutes with colimits and sends $\bS$ to the genuine sphere $\bS_G$. In particular, $X_G$ also admits a description as $X\otimes \bS_G$, using that $\Sp^G$ is tensored over spectra, since it is stable and admits colimits. This also shows that $X\mapsto X_G$ is symmetric-monoidal, so if $R$ is a commutative ring spectrum, $R_G$ is a commutative algebra in $\Sp^G$. Finally, we also see $\Phi^K X_G \simeq X$.

\begin{theorem}
\label{thm:pullbackcoeff}
Let $R$ be an ordinary connective commutative ring spectrum, and let $\cF$ and $K$ be as in Theorem \ref{thm:pullback}. Then there is a pullback diagram of symmetric-monoidal stable $\infty$-categories
\[
\begin{tikzcd}
\Mod_\omega(R_G)/\cF \rar\dar{\Phi^K} & \Mod_\omega(R_G)/\cF\cup\{G/K\}\dar{\overline{\Phi}^K}\\
\Fun(BW_G(K), \Mod_\omega(R)) \rar & \Fun(BW_G(K),\Mod_\omega(R))/\langle R[W_G(K)]\rangle,
\end{tikzcd}
\]
where $\Mod_\omega(R_G)/\cF$ denotes the Verdier quotient by $\langle R_G\otimes \Sigma^\infty_+ G/H\rangle_{G/H \in \cF}$.
\end{theorem}
\begin{proof}
The conclusion of Lemma \ref{lem:geomfixedpointsiso} still holds for the functor $\Mod_\omega(R)/\cF\to \Fun(BW_G(K), \Mod_\omega(\Phi^K R))$. For maps out of $R\otimes \Sigma^\infty_+ G/K$, this reduces by an adjunction to the claim that $\colim_{X\xto{\sim} X'} X'$ has trivial geometric $H$-fixed points for all $H\in \cF$, which is done by a similar argument as there. As objects $\Mod(R)_\omega$ are still dualizable, the rest of the argument goes through unchanged.

To see that $\Phi^K$ is essentially surjective on kernels, we can similarly construct a section: The functor $\Fun(BW_G(K),\Sp)\to \Sp^G$, $X\mapsto X\otimes_{W_G(K)} \Sigma^\infty_+ (G/K)$ sends $R$-modules to $R_G$-modules, giving a functor $\Fun(BW_G(K),\Mod(R)) \to \Mod(R_G)$ which sends $R[W_G(K)]\mapsto R_G \otimes \Sigma^\infty_+ (G/K)$.
After restricting to $\Fun(BW_G(K),\Mod_\omega(R))$, this produces a section to $\Phi^K: \Mod(R_G)_\omega/\cF \to \Fun(BW_G(K), \Mod_\omega(R))$, which sends $\langle R[W_G(K)]\rangle$ to $\langle \Sigma^\infty_+ G/K\rangle$. This shows that the restriction of $\Phi^K$ to the kernels is essentially surjective.
\end{proof}

In the next section, we will interpret the lower right corner as a variant of the stable module category of $G$.

\section{Stable module categories with general coefficients}

We present here a version of the stable module category of a group $G$ with coefficients in an arbitrary commutative ring spectrum $R$. For $R$ the Eilenberg-MacLane spectrum associated to a field $k$ of characteristic $p$, this recovers the usual notion.

\begin{definition}
\label{def:stmod}
For a group $G$ and commutative ring spectrum $R$, we define the \emph{small stable module category} $\StMod^{\fin}(RG)$ as the Verdier quotient
\[
\Fun(BG, \Mod_\omega(R))/\langle R[G]\rangle.
\]
\end{definition}

For $R$ an ordinary ring (where we still use $\Mod(R)$ to refer to module spectra over $R$ considered as an Eilenberg-Mac Lane spectrum), $\Mod_\omega(R)$ agrees with the full subcategory of the derived ($\infty$-)category of $R$ on perfect complexes.

Thus, objects of $\StMod^{\fin}(RG)$ are represented by perfect $R$-complexes with a $G$-action, and morphisms can be described by Lemma \ref{lem:mapsinquotient}. We first make this description explicit:

\begin{lemma}
\label{lem:mappingtate}
For $R$ a commutative ring spectrum and $X,Y$ compact $R$-modules with a $G$-action, we have that
\[
\map_{\StMod^{\fin}(RG)}(X,Y) = \map_{\Mod(R)}(X,Y)^{tG}.
\]
\end{lemma}
\begin{proof}
By Lemma \ref{lem:mapsinquotient}, we can describe $\map_{\StMod^{\fin}(RG)}(X,Y)$ as
\[
\colim_{Y\xto{\sim} Y'} \map_{\Fun(BG, \Mod(R)_\omega)}(X,Y') = \colim_{Y\xto{\sim} Y'} \map(X,Y')^{hG},
\]
which, by writing the maps $Y\to Y'$ as cofibers of their fibers $A\to Y\in \langle R[G]\rangle_{/Y}$, we can exhibit as the cofiber of
\begin{equation}
\label{eq:cofibermap}
\colim_{A \in \langle R[G]\rangle_{/Y}} \map_{\Mod(R)}(X,A)^{hG}\to \map_{\Mod(R)}(X,Y)^{hG}
\end{equation}
If $A$ is in $\langle R[G]\rangle$ and $X$ is dualizable, $\map_{\Mod(R)}(X,A) = DX\otimes_R A$ is annihilated by the Tate construction, so in the diagram
\[
\begin{tikzcd}
\colim_{A \in \langle R[G]\rangle_{/Y}} \map_{\Mod(R)}(X,A)_{hG}\dar{N}\rar& \map_{\Mod(R)}(X,Y)_{hG}\dar{N}\\
\colim_{A \in \langle R[G]\rangle_{/Y}} \map_{\Mod(R)}(X,A)^{hG}\rar& \map_{\Mod(R)}(X,Y)^{hG}\\
\end{tikzcd}
\]
the left vertical map is an equivalence. The upper horizontal map is also an equivalence: Because homotopy orbits commute with colimits, and $X$ is compact, it suffices to check that
\[
\colim_{A\in \langle R[G]\rangle_{/Y}} A \to Y
\]
is an equivalence. Equivalently, it suffices to check that the cofiber $\colim_{Y\xto{\sim} Y'} Y'$ is zero. We can check this on homotopy groups, where it follows from the fact that the cofiber of a map $\Sigma^n R[G]\to Y'$ is still an object in the colimit diagram.

This identifies the map in \eqref{eq:cofibermap} with the norm $N: \map_{\Mod(R)}(X,Y)_{hG} \to \map_{\Mod(R)}(X,Y)^{hG}$, and thus its cofiber with $\map_{\Mod(R)}(X,Y)^{tG}$, as desired.
\end{proof}

\begin{remark}
Some authors, e.g. \cite{mathew}, refer by stable module category instead to the large, presentable $\infty$-category $\Ind(\StMod^{\fin}(kG))$. The compact objects in there agree with the idempotent completion of $\StMod^{\fin}(kG)$. By a similar argument as in Proposition \ref{prop:idempotentcompletion}, one sees that this agrees with $\StMod^\fin(kG)$ in the case of a finite field.
In general, there are idempotents in $\StMod^{\fin}(RG)$ that do not split, this happens already for $R=\Z$, $G=C_6$.
\end{remark}

In the language of Definition \ref{def:stmod}, Theorem \ref{thm:pullback} thus exhibits $\Sp^G_\omega$ as an iterated pullback along functors $\Fun(BW_G(K), \Sp_\omega) \to \StMod^{\fin}(\bS G)$. In particular, we get a corresponding pullback for Picard groupoids. This is not yet very useful: Even $\Pic(\Fun(BW_G(K), \Sp_\omega))$ is complicated, consisting of all $W_G(K)$-actions on shifts of $\bS$. $\StMod^{\fin}(\bS G)$ is even more difficult to understand, as the Tate construction in spectra behaves very surprisingly, due to the Segal conjecture (proven by Carlsson in \cite{carlsson}).

While both categories are individually complicated, the problem becomes much simpler in a relative sense. More precisely, we will be able to control the above functor by the corresponding $\Z$-linear version $\Fun(BW_G(K), \Mod_\omega(\Z)) \to \StMod^{\fin}(\Z G)$. This will allow us to control the Picard group of $\Sp^G$ through the Picard group of $\StMod^{\fin}(\Z G)$, which is a purely algebraic object.

\begin{definition}
Let $R$ be a connective commutative ring spectrum. We say an object $X\in \Mod_\omega(R)$ is \emph{concentrated in $[a,b]$} for integers $a,b$ if $X$ is $a$-connective and the $R$-linear dual $DX$ is $(-b)$-connective.
\end{definition}

\begin{remark}
One can check that the notion of being concentrated in $[a,b]$ agrees with the notion of having Tor-amplitude in $[a,b]$, which is more generally defined for non-dualizable modules. (An object $X\in \Mod(R)$ has Tor-amplitude in $[a,b]$ if every $X\otimes_R M$ has homotopy groups concentrated in $[a,b]$ for every $R$-module $M$ which has homotopy groups concentrated in degree $0$.) Note that being concentrated in $[a,b]$ is not the same as being $a$-connective and $b$-coconnective: For example, the free module $R$ is concentrated in $[0,0]$ even if $R$ has positive homotopy groups. Even for $R$ an ordinary ring, the two notions differ, for example the $\Z$-module $\Z/2$ is not concentrated in $[0,0]$, as $D\Z/2 = \Z/2[-1]$.
\end{remark}

\begin{lemma}
\leavevmode
\label{lem:concentratedbasics}
\begin{enumerate}
\item If $X$ is concentrated in $[a,b]$ with $a>b$, $X\simeq 0$.
\item If $X$ is concentrated in $[a,a]$, it is a retract of a finite sum of copies of $\Sigma^a R$ (``finitely generated projective concentrated in degree $a$'').
\end{enumerate}
\end{lemma}
\begin{proof}
If $X$ is concentrated in $[a,b]$ with $a>b$, the connectivity of $X\otimes_R DX$ is $a-b>0$. So the unit $R\to X\otimes_R DX$ of the duality is zero for degree reasons, and thus $X\simeq 0$.

If $X$ is concentrated in $[a,a]$, $X\otimes DX$ is $0$-connective, and $\pi_0(X\otimes_R DX) \cong \pi_a(X)\otimes_{\pi_0(R)} \pi_{-a}(DX)$. So the duality between $X$ and $DX$ yields a duality between $\pi_a(X)$ and $\pi_{-a}(DX)$ in the abelian category of $\pi_0(R)$-modules. In particular, those are finitely generated projective. We now choose a retract $P$ of a finite sum of copies of $\Sigma^a R$ and a map $P\to X$ inducing an isomorphism on $\pi_a$. By duality, the map $\pi_{-a}(DX) \to \pi_{-a}(DP)$ is an isomorphism as well, and so the cofiber sequences
\begin{gather*}
P\to X \to C\\
DC \to DX \to DP
\end{gather*}
show that $C$ is concentrated in $[a+1,a]$, therefore vanishes by the first statement. So $P\simeq X$, which finishes the proof. 
\end{proof}

\begin{lemma}
\label{lem:objectsrepresentative}
Let $R$ be a connective ring spectrum. Every object of $\StMod^{\fin}(RG)$ admits a representative in $\Fun(BG, \Mod_\omega(R))$ whose underlying $R$-module is concentrated in $[0,0]$, i.e. a retract of a finite sum of copies of $R$ (``finitely generated projective and concentrated in degree $0$'').

An object $X\in \Fun(BG,\Mod_\omega(R))$ whose underlying $R$-module is concentrated in $[0,0]$ is zero in $\StMod^{\fin}(R(G))$ if and only if $X$ is a retract of a finite sum of copies of $R[G]$.
\end{lemma}
\begin{proof}
Assume we have $X\in \Fun(BG,\Mod_\omega(R))$ whose underlying $R$-module is concentrated in the interval $[a,b]$ with $a<0\leq b$. We find a map
\[
\bigoplus \Sigma^a R[G] \to X,
\]
from a finite free module on generators of degree $a$, which is surjective on $\pi_a$. Its cofiber $X'$ now is $a+1$-connective, and the dual fiber sequence
\[
DX' \to DX \to \Sigma^{-a} \bigoplus R[G]
\]
shows that, since $a<b$, $DX'$ is still $(-b)$-connective, so $X'$ is concentrated in $[a+1,b]$. Also, by construction, $X\to X'$ is an equivalence in $\StMod^{\fin}(RG)$.

Inductively, we can assume $X$ to be concentrated in $[0,b]$. So the dual is concentrated in $[-b,0]$, and we can perform the same construction as before, before dualizing back to obtain a representative of the same object of $\StMod^{\fin}(RG)$ concentrated in the interval $[0,0]$.

For the second part, assume $X\in \Fun(BG,\Mod_\omega(R))$ is zero in $\StMod^{\fin}(RG)$, i.e. $X\in \langle R[G]\rangle$.
Since $\map_{\Fun(BG,\Mod_{\omega}(R))}(R[G], R[G]) = R[G]$ is connective, a standard argument shows that every object of $\langle R[G]\rangle$ is a retract of a finite complex, i.e. there is a $C$ with a finite filtration by $C_n$ such that each cofiber $C_n/C_{n-1}$ is a finite sum of copies of $\Sigma^n R[G]$, and the identity of $X$ factors as $X\to C \to X$.

Now since $C/C_0$ is an extension of copies of $\Sigma^i R$ for $i>0$, and $DX$ is connective, so all maps $X\to \Sigma^i R$ for $i>0$ are nullhomotopic, the map $X\to C$ factors over $C_0$. We thus obtain a factorization of the identity on $X$ as $X\to C_0 \to X$. Since $C_{-1}$ is an extension of copies of $\Sigma^i R$ for $i<0$, we similarly have that the map $C_{-1}\to X$ is nullhomotopic, so the map $C_0\to X$ factors through $C_0/C_{-1}$, and we obtain a factorization of the identity on $X$ as $X\to C_0/C_{-1} \to X$. By assumption, $C_0/C_{-1}$ is a finite sum of copies of $R[G]$, as desired.
\end{proof}

\begin{remark}
Note that this shows in particular that every finite spectrum with $G$-action is equivalent in $\Fun(BG,\Sp_\omega)$ to some $G$-action on a finite wedge of spheres.
This highlights how far such actions are from being classified by their homology: For $G=C_p$, the Segal conjecture together with Lemma \ref{lem:objectsrepresentative} tells us that every finite $p$-complete spectrum arises as $tC_p$ applied to a suitable $C_p$-action on a wedge of $0$-spheres.
\end{remark}

In terms of representatives concentrated in $[0,0]$, we can also give a description of morphisms in $\StMod^{\fin}(RG)$:

\begin{lemma}
Let $X,Y\in \Fun(BG,\Mod_\omega(R))$ be two objects whose underlying $R$-module is concentrated in $[0,0]$. Then 
\[
[X,Y]_{\StMod^{\fin}(RG)} := \pi_0\map_{\StMod^{\fin}(RG)}(X,Y)
\]
is given by the group of maps $X\to Y$ in $\Fun(BG,\Mod_\omega(R))$ modulo those maps which factor through a finite sum of copies of $R[G]$.
\end{lemma}
\begin{proof}
We first check that a map $X\to Y$ in $\Fun(BG,\Mod_\omega(R))$ becomes zero in $\StMod^\fin(RG)$ precisely if it factors through a finite sum of copies of $R[G]$. By Lemma \ref{lem:mapsinquotient}, a map $X\to Y$ becomes zero in $\StMod^\fin(RG)$ precisely if we can precompose it with a mod $\langle R[G]\rangle$-equivalence $X'\to X$ such that the composite $X'\to X \to Y$ becomes zero. Equivalently, this means that a map is zero precisely if it factors as
\[
X \to A \to Y
\]
with $A\in \langle R[G]\rangle$.

Given such a diagram, assume $A$ is concentrated in $[a,b]$ with $a<0\leq b$. As in the proof of Lemma \ref{lem:objectsrepresentative}, we can take the cofiber $A'$ of a suitable map $\bigoplus \Sigma^a R[G]\to A$ to achieve that $A'$ is concentrated in $[a+1,b]$, and because of the connectivity of $Y$ the map $A\to Y$ extends over $A'$. Inductively, we can assume $A$ to be concentrated in $[0,b]$, and by performing the same argument on the dual diagram
\[
DY \to DA \to DX
\]
before dualizing back, we can actually assume $A$ to be concentrated in $[0,0]$. Since it is also zero in $\StMod^{\fin}(RG)$, we see by Lemma \ref{lem:objectsrepresentative} that $A$ is a retract of a finite sum of copies of $R[G]$, so $X\to Y$ also factors through such a finite sum as we wanted to prove.

We now show every morphism in $\StMod^\fin$ between objects concentrated in $[0,0]$ can be represented by an actual map in $\Fun(BG,\Mod_\omega(R))$.
By Lemma \ref{lem:mapsinquotient}, maps $X\to Y$ in $\StMod^\fin$ can be represented as zigzag
\[
\begin{tikzcd}
X\ar{dr} & & Y\ar{ld}{\sim}\\
 & Y' &
\end{tikzcd}
\]
of maps in $\Fun(BG,\Mod_\omega(R))$. Assume $Y'$ is concentrated in $[a,b]$ with $a<0< b$. By taking the cofiber of a suitable map $\bigoplus \Sigma R[G]\to Y'$ as in the proof of Lemma \ref{lem:objectsrepresentative}, we can replace $Y'$ by a new object, concentrated in $[a+1,b]$, without changing the morphism represented by the zigzag. Inductively, we may assume $Y'$ to be concentrated in $[0,b]$.

We now dualize to a diagram
\[
\begin{tikzcd}
 & DY'\ar{dl}{\sim}\ar{dr} & \\
DY & & DX\\
\end{tikzcd}
\]
with $DY'$ concentrated in $[-b,0]$ with $-b<0$. In particular, the fiber $F$ of $DY' \to DY$ is concentrated in $[-b,0]$. By again choosing a suitable map $\bigoplus\Sigma^{-b} R[G]\to F$ and taking the cofiber of the composite $\bigoplus \Sigma^{-b}R[G]\to DY'$, we can change the situation in order to make the fiber $F$ concentrated in $[-b+1,0]$. Inductively, we may assume the fiber of $DY'\to DY$ to be concentrated in $[0,0]$.
Dualizing back, we have represented the original morphism by a zigzag $X\to Y'\xleftarrow{\sim} Y$, where the right-hand map has cofiber concentrated in $[0,0]$.

This cofiber $P$ is also zero in $\StMod^{\fin}(RG)$, so it is a retract of a finite sum of copies of $R[G]$, and so we can split the cofiber sequence $Y\to Y' \to P$ in $\Fun(BG,\Mod_\omega(R))$. The original morphism is therefore represented by a zigzag $X\to Y\oplus P \leftarrow Y$. The composite $X\to Y\oplus P \to Y$ represents the same morphism, since in $\StMod^{\fin}(RG)$, the inclusion $Y\to Y\oplus P$ and the projection $Y\oplus P\to Y$ are inverse equivalences.
\end{proof}

We can also derive a uniqueness statement for representatives:

\begin{lemma}
\label{lem:representativesuniqueness}
Given $X,Y\in \Fun(BG,\Mod_\omega(R G))$ concentrated in $[0,0]$ and an equivalence $X\to Y$ in $\StMod^\fin(R G)$, there exist $P, Q$ both retracts of finite sums of $R[G]$, and an isomorphism $X\oplus P \simeq Y\oplus Q$ in $\Fun(BG,\Mod_\omega(R G))$ inducing the given equivalence.
\end{lemma}
\begin{proof}
We can represent the given equivalence by a map $X\to Y$. Picking a suitable sum $Q$ of copies of $R[G]$, and a map $X\to Q$ inducing an injection on $\pi_0$ (which we can do by dually surjecting onto $\pi_0(DX)$), we obtain a map $X\to Q \oplus Y$ which is injective on $\pi_0$. As a result, its cofiber $P$ is concentrated in $[0,0]$, and zero in $\StMod^\fin(R G)$ since the map $X\to Y$ was an equivalence. So $P$ is a retract of a sum of copies of $R[G]$, and thus the cofiber sequence
\[
X \to Q \oplus Y \to P
\]
splits, yielding the desired result.
\end{proof}

We now explain how the stable module category with coefficients in the sphere spectrum $\bS$ relates to the stable module category with coefficients in $\Z$. Note that we have a functor $\Mod_\omega(\bS)\to \Mod_{\omega}(\bZ)$ given by basechange, which induces a functor $\StMod^{\fin}(\bS G) \to \StMod^{\fin}(\bZ G)$. More generally, for a connective commutative ring spectrum $R$, we have a basechange functor from $R\to \pi_0(R)$. In analogy with the situation for $R=\bS$, where this basechange corresponds to taking homology (or more precisely, chains), we will denote this functor simply by $H$.  

\begin{lemma}
\label{lem:liftinghomology}
Let $R$ be connective and let $X$ be an object of $\StMod^{\fin}(R G)$, let $M\in \Fun(BG,\Mod_\omega(\pi_0(R)))$ be concentrated in degree $[0,0]$, i.e. a finitely generated projective $\pi_0(R)$-module with a $G$-action, and assume we have an equivalence $M \simeq HX$ in $\StMod^{\fin}(\pi_0(R) G)$. Then there exists a representative $X'\in\StMod^\fin(R G)$ of $X$ with $HX'\simeq M$ in $\Fun(BG,\Mod_\omega(\pi_0(R)))$. More precisely, there exists $X'\in \Fun(BG,\Mod_{\omega}(R))$ with $HX'\simeq M$ in $\Fun(BG,\Mod_{\omega}(\pi_0(R)))$ and an equivalence $X'\simeq X$ in $\StMod^{\fin}(R G)$ lifting the given equivalence $M\simeq HX$.
\end{lemma}
\begin{proof}
Without limiting generality, we can assume $X$ to be already represented by an object of $\Fun(BG,\Mod_\omega(R))$ concentrated in $[0,0]$. Then $HX$ is also concentrated in $[0,0]$, and the given equivalence $M\to HX$ comes from an isomorphism in $\Fun(BG, \Mod_\omega(\pi_0(R)))$ 
\[
P\oplus M \to Q \oplus HX.
\]
We thus have a map $P\to Q\oplus HX$ with cofiber $M$. We now observe that we can lift the objects $P$ and $Q$ to objects of $\Fun(BG,\Mod_\omega(R))$. Namely, free modules $\bigoplus \pi_0(R)[G]$ clearly lift to free modules $\bigoplus R[G]$, and all maps between them lift uniquely up to homotopy due to a version of Hurewicz. So if $P$ and $Q$ are given as retracts of free $\pi_0(R)[G]$ modules, we can lift the corresponding idempotent and obtain lifts $\widetilde{P}$ and $\widetilde{Q}$. Similarly, homotopy classes of maps
\[
\widetilde{P}\to \widetilde{Q}\oplus X
\]
correspond precisely to maps $P\to Q \oplus HX$. We can first check this more general statement for $P$ a free module, where it follows by a Hurewicz argument, and then pass to retracts.

We now let $X'$ be the cofiber of the unique lift $\widetilde{P}\to \widetilde{Q} \oplus X$, this has the desired properties.
\end{proof}

We can express this previous fact in a more categorical way:

\begin{definition}
We call a diagram
\[
\begin{tikzcd}
\cA\rar\dar& \cB\dar\\
\cC \rar & \cD
\end{tikzcd}
\]
of $\infty$-categories $1$-cartesian if the functor $\cA \to \cC\times_\cD \cB$ is essentially surjective and $1$-connective on mapping spaces, i.e. if the maps
\[
\Map_{\cA}(X,Y) \to \Map_{\cC\times_{\cD} \cB}(X,Y)
\]
are isomorphisms on $\pi_0$ and surjective on $\pi_1$. In particular, a $1$-cartesian diagram of $\infty$-categories leads to a pullback diagram of their homotopy categories.
\end{definition}

\begin{proposition}
\label{prop:pullbackdegreen}
Let $R$ be a connective commutative ring spectrum, and let $\Mod^n_\omega(R)$ denote the full subcategory of $\Mod_\omega(R)$ on objects concentrated in $[n,n]$. Then the diagram
\[
\begin{tikzcd}
\Fun(BG, \Mod^n_\omega(R)) \rar\dar & \StMod^\fin(R G)\dar\\
\Fun(BG, \Mod^n_\omega(\pi_0(R)))\rar & \StMod^\fin(\pi_0(R) G)
\end{tikzcd}
\]
is $1$-cartesian.
\end{proposition}
\begin{proof}
By applying a suitable shift, we can assume $n=0$.
The functor from the upper left corner into the actual pullback is essentially surjective: This is precisely the statement of Lemma \ref{lem:liftinghomology}. For the condition on mapping spaces we need to show that the total homotopy fiber of
\[
\begin{tikzcd}
\map_{\Fun(BG, \Mod^n_{\omega}(R))}(X,Y) \rar\dar & \map_{\StMod^{\fin}(R G)}(X,Y)\dar\\
\map_{\Fun(BG, \Mod^n_{\omega}(\pi_0(R))}(HX,HY) \rar & \map_{\StMod^{\fin}(\pi_0(R) G)}(HX,HY)
\end{tikzcd}
\]
is connected.
By Lemma \ref{lem:mappingtate}, the horizontal fibers are given by $\map_R(X,Y)_{hG}$ and $\map_{\pi_0(R)}(HX,HY)_{hG}$ respectively. Since homotopy orbits preserve connectivity, it is now sufficient to check that the fiber of
\[
\map_R(X,Y) \to \map_{\pi_0(R)}(HX,HY)
\]
has connected fiber. But since $X$ and $Y$ are retracts of sums of copies of $R$, this follows at once from the fact that the fiber of $R\to \pi_0(R)$ is connected.
\end{proof}

While this does not control the entire category $\Fun(BG, \Mod_\omega(R))$, it does control its Picard groupoid. We first observe that as long as $\pi_0(R)$ is a domain, each invertible object in $\Fun(BG, \Mod_{\omega}(R))$ is concentrated in a single degree, this is a well-known fact about Picard groups of connective ring spectra which we recall here:

\begin{lemma}
\label{lem:picardmodr}
For a connective commutative ring spectrum $R$ whose $\pi_0(R)$ is a domain,
let $X\in \Fun(BG,\Mod_\omega(R))$ be invertible. Then the underlying $R$-module of $X$ is concentrated in $[n,n]$ for some $n$.
\end{lemma}
\begin{proof}
Since $R\to \pi_0(R)$ has $1$-connective fiber, for any $R$-module $X$ the connectivity of $X$ agrees with the connectivity of the $\pi_0(R)$-module $\pi_0(R)\otimes_R X$. In particular, $X$ is concentrated in $[a,b]$ if and only if the basechange $\pi_0(R)\otimes_R X$ is concentrated in $[a,b]$ as a $\pi_0(R)$-module. We can therefore reduce to the case of $R$ an ordinary ring.

We now first check the claim for a \emph{local} ring $R$, with maximal ideal $\mathfrak{m}$. Let $X$ be invertible, and let $a$ be maximal and $b$ minimal with $X$ concentrated in $[a,b]$, i.e. $\pi_a(X)\neq 0$ and $\pi_{-b}(DX)\neq 0$. We want to show $a=b$. Assume on the contrary that $a<b$. Then $X\otimes DX \simeq R$, and so
\[
0 \simeq \pi_{a-b}(R) \simeq \pi_{a-b}(X\otimes DX) \simeq \pi_a(X)\otimes_{R} \pi_{-b}(DX).
\]
But over local rings the tensor product of modules has no zero divisors, which one can check after basechange to the field $R/\mathfrak{m}$ because of the Nakayama lemma. So $\pi_a(X)=0$ or $\pi_{-b}(DX)=0$, a contradiction.

Now let $R$ be an ordinary ring, which is a domain but not necessarily local, and let $X$ be invertible. The basechange to each localisation $R_{(\mathfrak{m})}$ is concentrated in $[n_{\mathfrak{m}}, n_{\mathfrak{m}}]$ for some $n_{\mathfrak{m}}$, by basechanging further to the fraction field $R_{(0)}$ we see that all $n_{\mathfrak{m}}$ agree, call that number $n$. It follows that the basechange of $X$ to each localization $R_{\mathfrak{m}}$ is concentrated in $[n,n]$. This means that $\pi_i(X)$ for each $i<n$ is an $R$-module which vanishes after basechange to each $R_{(\mathfrak{m})}$, i.e. it is zero. The same argument shows that $\pi_i(DX)=0$ for $i<-n$. So $X$ is concentrated in $[n,n]$.
\end{proof}

\begin{theorem}
\label{thm:pic1cartesian}
Let $R$ be a connective commutative ring spectrum with $\pi_0(R)$ a domain.
We have the following $1$-cartesian diagram of Picard groupoids:
\[
\begin{tikzcd}
\cPic(\Fun(BG, \Mod_\omega(R))) \rar\dar & \cPic(\StMod^\fin(R G))\dar\\
\cPic(\Fun(BG, \Mod_\omega(\pi_0(R))))\rar & \cPic(\StMod^\fin(\pi_0(R) G))
\end{tikzcd}
\]
\end{theorem}
\begin{proof}
Every invertible object of $\Mod_\omega(R)$ is concentrated in $[n,n]$ for some $n$ by Lemma \ref{lem:picardmodr}. It follows that the Picard groupoids of $\Fun(BG, \Mod_\omega(R))$ and $\Fun(BG, 
\Mod_\omega(\pi_0(R))$ are contained in the symmetric-monoidal subcategories $\coprod_n \Fun(BG, \Mod^n_{\omega}(R))$ and $\coprod_n \Fun(BG, \Mod^n_{\omega}(\pi_0 R))$ (and therefore are also given as $\cPic$ of those). So the claim follows from the $1$-cartesian symmetric-monoidal diagram
\[
\begin{tikzcd}
\coprod_n \Fun(BG, \Mod^n_\omega(R)) \rar\dar & \StMod^\fin(R G)\dar\\
\coprod_n \Fun(BG, \Mod^n_\omega(\pi_0(R)))\rar & \StMod^\fin(\pi_0(R) G)
\end{tikzcd}
\]
obtained from Proposition \ref{prop:pullbackdegreen}.
\end{proof}

We obtain the following result controlling the Picard groupoid of $\Sp^G$:

\begin{theorem}
\label{thm:pictotal1cartesian}
Let $\cF$ be a family of orbits of $G$, and let $G/K\not\in\cF$ with $G/K'\in\cF$ for all $K'\subsetneq K$. Then we have a $1$-cartesian diagram
\[
\begin{tikzcd}
\cPic(\Sp_\omega^G/\cF) \rar\dar{H\Phi^K}& \cPic(\Sp^G_\omega/\cF\cup\{G/K\})\dar{H\overline{\Phi}^K}\\
\cPic(\Fun(BW_G(K), \Mod_\omega(\Z)))\rar & \cPic(\StMod^\fin(\Z W_G(K)))
\end{tikzcd}
\]
\end{theorem}
\begin{proof}
The pullback diagram from Theorem \ref{thm:pullback} yields a pullback diagram after applying $\cPic$, since $\cPic$ preserves limits, see \cite[Prop. 2.2.3]{mathewstojanoska}.
We can paste this with the $1$-cartesian diagram from Theorem \ref{thm:pic1cartesian}, resulting in the desired diagram.
\end{proof}

In particular, an invertible element $\widetilde{X}$ of $\Sp^G/\cF$ is given precisely by an invertible element $X$ of $\Sp^G/\cF\cup\{G/K\}$, a candidate for the homology $H\Phi^K(\widetilde{X})$ of the geometric fixed points, and an equivalence in $\StMod^\fin(\Z W_G(K))$ between the images. The objects of $\cPic(\Fun(BW_G(K), \Mod_\omega(\Z)))$ are, for example by Lemma \ref{lem:picardmodr}, precisely classified by a dimension and a homomorphism $W_G(K)\to \{\pm 1\}$ describing the action of $W_G(K)$ on $\Z[n]$. The functor $\overline{\Phi}^K$ can be described explicitly on an object $X\in \Sp^G_\omega/\cF\cup\{G/K\}$ as choosing some (not necessarily invertible) lift $\widetilde{X}\in\Sp_\omega^G/\cF$ and then considering the homology $H\Phi^K \widetilde{X}$ as an object in $\StMod^\fin(\Z W_G(K))$.

For $R_G$ as in Theorem \ref{thm:pullbackcoeff}, we admit a similar statement:

\begin{theorem}
Assume $R$ is connective and $\pi_0(R)$ is a domain. Then, with assumptions as in Theorem \ref{thm:pictotal1cartesian}, we have a $1$-cartesian diagram
\[
\begin{tikzcd}
\cPic(\Mod_\omega(R_G)/\cF) \rar\dar{H\Phi^K}& \cPic(\Mod_\omega(R)/\cF\cup\{G/K\})\dar{H\overline{\Phi}^K}\\
\cPic(\Fun(BW_G(K), \Mod_\omega(\pi_0(R))))\rar & \cPic(\StMod^\fin(\pi_0(R) W_G(K)))
\end{tikzcd}
\]
\end{theorem}

We are now able to prove Theorem \ref{thm:linearisationintro}.
\begin{theorem}
\label{thm:linearisation}
The map $\Pic(\Sp^G) \to \Pic(\Mod(H\Z_G))$ is an isomorphism.
\end{theorem}
\begin{proof}
We prove inductively that $\Pic(\Sp^G_\omega/\cF)\to \Pic(\Mod_\omega(H\Z_G)/\cF)$ is an isomorphism. This is clear, since in the natural $1$-cartesian diagram of Theorem \ref{thm:pictotal1cartesian} objects on both sides are described by the same data.
\end{proof}

\section{Localisation}
In this section, we discuss how the stable module categories $\StMod^\fin(\Z G)$ are related to completed variants $\StMod^\fin(\Z_p G)$ and subgroups of $G$.

\begin{lemma}
Let $G$ be a finite group and $n=|G|$ its order. Let $\Z_n$ denote the completion $\lim_k \Z/n^k$ (this is a product of $\Z_p$  where $p$ ranges over all primes dividing $n$.)
The functor
\[
\StMod^\fin(\Z G)\to \StMod^\fin(\Z_n G) \simeq \prod_{p\mid n}\StMod^\fin(\Z_p G)
\]
is fully faithful.
\end{lemma}
\begin{proof}
We can identify the map
\[
\map_{\StMod^{\fin}(\Z G)}(X,Y) \to \map_{\StMod^{\fin}(\Z_n G)}(\Z_n \otimes_{\Z} X,\Z_n \otimes_{\Z} Y)
\]
with the map
\[
(DX\otimes_{\Z} Y)^{tG} \to (\Z_n \otimes_{\Z} DX\otimes_{\Z} Y)^{tG}.
\]
Both sides are $n$-complete since they are modules over $\Z^{tG}$, $\pi_0$ of which is $\Z/n$. So we can check equivalence after reducing mod $n$, where it follows from the fact that $\Z$ and $\Z_n$ agree after reduction mod $n$. 
\end{proof}

\begin{proposition}
\label{prop:idempotentcompletion}
The fully faithful functor
\[
\StMod^\fin(\Z G)\to \StMod^\fin(\Z_n G) \simeq \prod_{p\mid n}\StMod^\fin(\Z_p G)
\]
is an idempotent completion, i.e. its codomain is idempotent complete, and every object is a retract of an object in the essential image.
\end{proposition}
\begin{proof}
We first show every object in the codomain is a retract of one in the essential image. Multiplication by $n$ is zero in $\StMod^\fin(\Z_n G)$, and so every object $Y$ is a retract of $Y/n$. Representing $Y$ by a free $\Z_n$-module concentrated in degree $0$, we see that $Y/n$ is represented by a free $\Z/n$-module in degree $0$, which is clearly in the image of the basechange $\Z\to \Z_n$.

We now check that the codomain is indeed idempotent complete. It suffices to check that $\StMod^\fin(\Z_p G)$ is idempotent complete. Let $X\in\StMod^\fin(\Z_p G)$ be some object with idempotent. We can represent this by a diagram
\[
X \xto{\varepsilon} X
\]
in $\Fun(BG, \Mod_\omega(\Z_p))$, where $X$ is concentrated in $[0,0]$ and $\varepsilon$ is some not necessarily idempotent map which reduces to the given idempotent in $\StMod^\fin(\Z_p G)$.

Note that $\End_{\Z_p[G]}(X) \to \End_{\StMod^\fin(\Z_p G)}(X)$ annihilates all multiples of $p^v$, where $p^v$ is the largest power of $p$ dividing $|G|$. We can assume $\varepsilon$ to be an idempotent in $\End_{\Z_p[G]}(X)/p^v$. Indeed, the latter one is a finite ring, and so the values of $\varepsilon^k$ become eventually periodic, so for large enough $k$, $\varepsilon^{2k} = \varepsilon^k$ in $\End_{\Z_p[G]}(X)/p^v$, and so $\varepsilon^k$ becomes an idempotent mod $p^v$ that still reduces to our original idempotent in $\StMod^\fin(\Z_p G)$.

We claim we find compatible idempotents $\varepsilon_i\in \End_{\Z_p[G]}(X)/p^i$ for all $i\geq v$. Given $\varepsilon_i$, we have
\[
\varepsilon_i^2 = \varepsilon_i + a
\]
with $a\in p^i\End_{\Z_p[G]}(X)$ and commuting with $\varepsilon$. So 
\[
\varepsilon_i^{2p} = \varepsilon_i^p
\]
modulo $p^{i+1}$, and we can take $\varepsilon_{i+1} = \varepsilon_i^p$ as desired. In the limit, the $\varepsilon_i^p$ now define an idempotent element of $\End_{\Z_p[G]}(X)$ lifting the original idempotent. Since $\Fun(BG, \Mod_{\omega}(\Z_p))$ is idempotent complete, the result follows.
\end{proof}

\begin{corollary}
For $\cF, K$ as in Theorem \ref{thm:pullback},
the following diagram is also $1$-cartesian:
\[
\begin{tikzcd}
\cPic(\Sp^G_{\omega}/\cF) \rar\dar& \cPic(\Sp^G_{\omega}/\cF\cup \{G/K\})\dar\\
\cPic(\Fun(BW_G(K), \Mod_\omega(\Z))) \rar & \prod_{p\mid |G|}\cPic(\StMod^\fin(\Z_p G)))
\end{tikzcd}
\]
\end{corollary}

This only uses the fully faithfulness established in Proposition \ref{prop:idempotentcompletion}. In this form, we say that invertible objects of $\Sp^G/\cF$ consist of an invertible object of $\Sp^G_\omega/\cF\cup\{G/K\}$, a candidate for the homology of the geometric $K$-fixed points, and gluing data chosen locally at each $p$.

We now want to understand $\cPic(\StMod^\fin(\Z_p G))$ better. In the classical case with coefficients in a field $k$ of characteristic $p$, Grodal \cite{grodal} establishes a powerful descent mechanism that completely characterizes $\cPic(\StMod^\fin(k G))$ in terms of the $p$-subgroups of $G$. The approach there is quite explicit, but unfortunately generalizing it directly to $\Z_p$-coefficients seems to involve additional subtleties. 

We learned from Mathew that this descent result can be obtained from the general machinery developed in \cite{mathewnaumannnoel}. 
Specifically, the stable module category over a ring $R$ can be written as category of modules in $\Sp^G$ over the genuine Tate spectrum whose $H$-fixed points are $R^{tH}$. For $R=\Z_p$ (or $k$), this spectrum is $\cF$-nilpotent for the family of $p$-subgroups, and so
\begin{equation}
\label{eq:limstmod}
\Ind(\StMod^\fin(R G)) = \lim_{G/H \in \Orb_p(G)} \Ind(\StMod^\fin(R H))
\end{equation}
by \cite[Theorem 6.42.]{mathewnaumannnoel}.
Now for $H=e$, $\StMod^\fin(R H)=0$, and so one sees that one can drop the free orbit $G/e\in \Orb_p(G)$ from the above limit.

One then has two important formal consequences:
\begin{proposition}
We have
\[
\cPic(\StMod^\fin(\Z_p G)) = \lim_{G/H\in\Orb_p^*(G)} \cPic(\StMod^\fin(\Z_p H)).
\]
where $\Orb_p^*(G)$ denotes the full subcategory of orbits $G/H$ where $H$ is a nontrivial $p$-group.
\end{proposition}
\begin{proof}
$\cPic$ preserves limits. From \eqref{eq:limstmod}, we thus see that $\cPic(\Ind(\StMod^\fin(\Z_p G)))$ admits a corresponding limit description. Finally, $\cPic(\Ind(\StMod^\fin(\Z_p G))) = \cPic(\StMod^\fin(\Z_p G))$, since for $R=\Z_p$ the category $\StMod^\fin(\Z_p G)$ is idempotent complete by Proposition \ref{prop:idempotentcompletion}, and so agrees with compact objects of its $\Ind$-category.
\end{proof}

\begin{proposition}
\label{prop:stmodstable}
Morphisms in $\StMod(\Z_p G)$ admit a \emph{stable element description} (as in group cohomology, cf. \cite[Theorem XII.10.1]{cartaneilenberg}). More precisely, $[X,Y]_{\StMod(\Z_p G)}$ agrees with the subgroup of all elements of $[X,Y]_{\StMod(\Z_p P)}$ ($P$ a $p$-Sylow subgroup) which are stable in the following sense: Given $P'\subseteq P$ and $g\in G$ with $(P')^g = gP'g^{-1}\subseteq P$, the two images of the morphism $X\to Y$ in $\StMod(\Z_p P)$ under the diagram
\[
\begin{tikzcd}
\StMod(\Z_p P) \rar\ar{rd} & \StMod(\Z_p P')\dar{\sim}\\
                         &  \StMod(\Z_p (P')^g)
\end{tikzcd}
\]
agree.
\end{proposition}

The fact that the stable module category can be identified with a module category in $\Sp^G$ over a genuine Tate spectrum seems to be well-known, but we reproduce it here since we were unable to find a reference.

\begin{lemma}
\label{lem:generatorsstmod}
Let $R=\Z$ or $\Z_p$.

$\StMod^\fin(R G)$ is generated by the objects $R[G/H]$ as a thick subcategory, where $H$ ranges over subgroups. If $R$ is $\Z_p$, it suffices to let $H$ range over $p$-subgroups.
\end{lemma}
\begin{proof}
We need to show that every object can be obtained through retracts, extensions, cofibers and shifts from objects of the form $R[G/H]$. Recall first that the functor
\[
\StMod^\fin(\Z G) \to \prod_p \StMod^\fin(\Z_p G)
\]
is an idempotent completion by Proposition \ref{prop:idempotentcompletion}, in particular it is fully faithful and every object on the left is a retract of an object on the right. It is thus sufficient to check that every object on the right is a retract of a finite filtered object whose associated graded consists of sums of shifts of $R[G/H]$ (for various $H$), and thus it is sufficient to check this for $R=\Z_p$. 
In that situation, first assume $G$ is a $p$-group. In that case, every object $Y$ is a retract of $Y/p^k$ for suitable $k$, and $Y/p^k$ admits a finite length filtration with associated gradeds of the form $Y/p$. By choosing a representative of $Y$ concentrated in $[0,0]$ we may assume $Y/p$ to be represented by a finite $\F_p$-vector space with $G$-action. This in turn admits a finite length filtration of objects of the form form $\F_p[G/G]$, since every $\F_p[G]$-module admits a filtration by trivial representations. Since $\F_p[G/G]$ admits (as a complex) a length $2$ filtration with associated graded $\Z_p[G/G]$ and $\Sigma^1\Z_p[G/G]$, there is a combined filtration on $Y$ whose associated graded consists of shifts of $\Z_p[G/G]$.

For general $G$ and $\Z_p$-coefficients, we may use the corresponding fact for a $p$-Sylow subgroup $P$ together with the observation that every $\Z_p[G]$-module $Y$ is a retract of $\Z_p[G]\otimes_{\Z_p[P]} Y$.
\end{proof}

\begin{proposition}
\label{prop:stmodmodules}
Let $R$ be $\Z$ or $\Z_p$.

There is an equivalence
\[
\Ind(\StMod^\fin(R G)) \simeq \Mod_{\Sp^G}(R^{t_{\gen} G})
\]
where $R^{t_{\gen}G}\in \Sp^G$ has $H$-fixed points $R^{tH}$ for each $H\subseteq G$.
\end{proposition}
\begin{proof}
We consider the symmetric monoidal functor
\[
\Sp^G_{\omega} \to \StMod^\fin(R G)
\]
obtained as the composite $\Sp^G_\omega \to \Fun(BG, \Sp_\omega) \to \Fun(BG, \Mod_\omega(R))\to \StMod^\fin(R G)$. It is exact and therefore induces a colimit-preserving symmetric monoidal functor
\[
\Sp^G \to \Ind(\StMod^\fin(R G)).
\]
By the adjoint functor theorem, it admits a right adjoint $\Ind(\StMod^\fin(R G)) \to \Sp^G$, which is then lax symmetric monoidal. The image of the unit is therefore a commutative algebra $R^{t_\gen G}\in \Sp^G$. Observe also that the left adjoint sends $\Sigma^\infty_+ G/H \mapsto R[G/H]$, so $R^{t_\gen G}$ has $H$-fixed points given by 
\[
\map_{\StMod^\fin(R G)}(R[G/H], R) \simeq \map_{\StMod^\fin(R H)}(R, R) \simeq R^{tH}.
\]
We obtain a factorisation 
\[
\Ind(\StMod^\fin(R G))\to \Mod_{\Sp^G}(R^{t_\gen G}) \to \Sp^G.
\]
Since the composite preserves limits and is accessible (in fact, colimit preserving, since the left adjoint $\Sp^G \to \Ind(\StMod^\fin(R G))$ preserves compacts), and the forgetful functor $\Mod_{\Sp^G}(R^{t_\gen G})\to \Sp^G$ preserves limits and colimits, the functor $\Ind(\StMod^\fin(R G)) \to \Mod_{\Sp^G}(R^{t_\gen} G)$ also preserves limits and is accessible, so it admits a left adjoint. By uniqueness of adjoints and the fact that adjoints compose, we have thus factored the original functor into
\[
\Sp^G \to \Mod_{\Sp^G}(R^{t_\gen G})\to \Ind(\StMod^\fin(R G)).
\]
It is now enough the check that the second functor is fully faithful, since it hits compact generators by Lemma \ref{lem:generatorsstmod}. We can check fully faithfullness on generators, and by duality, it suffices to analyze the case
\[
\map_{\Mod_{\Sp^G}(R^{t_\gen G})}(R^{t_\gen G}[G/H], R^{t_\gen G}) \to \map_{\StMod(R G)}(R[G/H], R).
\]
But this is an equivalence more or less by definition: The left hand side is given by $H$-fixed points of $R^{t_\gen G}$, which is the right hand side by construction of $R^{t_\gen G}$.
\end{proof}

\section{Dimension functions}
\label{sec:dimensionfunctions}
For an invertible object $X\in \cPic(\Sp^G/\cF)$, we have its \emph{dimension function} 
\[
\dim(X)\in C(G) = \prod_{G/H\in \Orb(G)} \Z
\]
and its \emph{orientation behaviour}
\[
e(X)\in \prod_{G/H\in \Orb(G)} \Hom(W_G(H), \{\pm 1\}).
\]
The former assigns to every $G/H$ the dimension of $\Phi^H X$, i.e. the integer $\dim_H(X)$ with $\Phi^H X \simeq S^{\dim_H(X)}$, the latter assigns to $X$ the action of $W_G(H)$ on $\Z \simeq H_{\dim_H(X)}(\Phi^H X)$. We think of this pair as \emph{generalized dimension function} of $X$, taking values in
\[
\cD(G) := \prod_{G/H \in \Orb(G)} \Z \times \Hom(W_G(H), \{\pm 1\}).
\]
For a generalized dimension function $d$, we write $d_H$ for the element of $\Z\times \Hom(W_G(H), \{\pm 1\})$, and $\Z[d_H]$ for the object of $\Fun(BW_G(H), \Mod_\omega(\Z))$ given by $\Z$ in degree given by the first factor of $d_H$, and $W_G(H)$-action given by the second factor of $d_H$. Note that every invertible object of $\Fun(BW_G(H), \Mod_\omega(\Z))$ is of this form by Lemma \ref{lem:picardmodr}. We will sometimes also abusively write $d_H$ for just the ``dimension'' part of $d_H$.

The kernel of the homomorphism
\[
\Pic(\Sp^G) \to C(G)
\]
is seen by tom Dieck-Petrie in \cite{tomdieckpetrie} to consist of invertible objects with trivial orientation behaviour, an observation we will be able to reproduce in Corollary \ref{cor:orientationbehaviour}.
Furthermore, they show that this kernel agrees with the Picard group of the Burnside ring $A(G)$ (see also \cite{fausklewismay}). More explicitly, for an object $X$ with trivial generalized dimension function, they check that the image of
\[
\pi_0^G(X)\xto{\Phi^H} [\Phi^H S^0, \Phi^H X]\simeq \Z \to \Z/|W_G(K)|
\]
consists of a single invertible element, well-defined up to sign (since the identification with $\Z$ depends on an orientation of $\Phi^HX$). $X$ is trivial precisely if this unit is $\pm 1$ for each $H$, and each unit can be realized. In particular, the kernel of $\Pic(\Sp^G)\to \cD(G)$ is given by the product of $(\Z/|W_G(H)|)^\times / \{\pm 1\}$ over all conjugacy classes of subgroups $H\subseteq G$.

In light of the $1$-cartesian diagrams provided by Proposition \ref{prop:pullbackdegreen} and Theorem \ref{thm:pictotal1cartesian}, this characterization arises very naturally.

\begin{definition}
We let
\[
\cD_\cF := \prod_{G/H \not\in \cF} \Z \times \Hom(W_G(H), \{\pm 1\})
\]
be the group of \emph{partial generalized dimension functions} (away from $\cF$), and $\Pic^0(\Sp^G_\omega/\cF)$ the kernel of the homomorphism
\[
\Pic(\Sp^G_\omega/\cF) \to \cD_\cF.
\]
\end{definition}

\begin{lemma}
\label{lem:liftingpic0}
Let $X\in \Sp^G_\omega/\cF\cup\{G/K\}$ be such that $X$ represents an object of $\Pic^0(\Sp^G_\omega/\cF\cup\{G/K\})$ (i.e. the partial generalized dimension function of $X$ is trivial). Then: 
\begin{enumerate}
\item There exists a map $S^0 \to X$ in $\Sp^G_\omega/\cF\cup\{G/K\}$ which induces an equivalence mod $|G|$.
\item The image $H\overline{\Phi}^K X$ in $\StMod(\Z W_G(K))$ is equivalent to $\Z[0]$. Here $H\overline{\Phi}^K$ denotes the right vertical map from the diagram of Theorem \ref{thm:pictotal1cartesian}.
\item There exist precisely $|(\Z/|W_G(K)|)/\{\pm 1\}|$ preimages of $X$ in $\Pic^0(\Sp^G_\omega/\cF)$.
\end{enumerate}
\end{lemma}
\begin{proof}
Let $\cF'$ denote $\cF\cup\{G/K\}$. We prove the lemma by induction on the size of $\cF'$. If $\cF'$ is all of $\Orb(G)$ (and $K$ necessarily $G$, $\cF$ the family of all nontrivial orbits of $G$), the first two statements are trivial, the third one is just the observation that an invertible object of $\Sp^G_\omega/\cF\simeq \Sp_\omega$ of dimension $0$ is equivalent to $S^0$.

For the inductive step, let $K'$ be such that $G/K'\not\in \cF'$ and that $\cF'\cup \{G/K'\}$ is still a family. Inductively, we may assume the statement of the lemma for $\cF'\cup \{G/K'\}$. Thus, the image of $X$ in $\Sp^G/\cF'\cup\{G/K'\}$ receives a map from $S^0\to X$ which is a mod $|G|$ equivalence. Proposition \ref{prop:pullbackdegreen} now says that in the square 
\[
\begin{tikzcd}
\map_{\Sp^G_\omega/\cF'}(S^0, X) \rar\dar & \map_{\Sp^G_\omega/\cF'\cup\{G/K'\}}(S^0, X)\dar\\
\map_{\Fun(BW_G(K'), \Mod_\omega(\Z))}(\Z[0],\Z[0]) \rar & \map_{\StMod(\Z W_G(K'))}(\Z[0],\Z[0])
\end{tikzcd}
\]
the map from the upper left corner to the pullback is an isomorphism on $\pi_0$. The given map $S^0\to X$ in $\Sp^G_\omega/\cF'\cup\{G/K'\}$ becomes an equivalence after reduction mod $|G|$ by assumption, and since the lower right corner has $\pi_0$ given by $\Z/|W_G(K')|$, in which $|G|$ is $0$, this shows that the image in the lower right corner is a unit of $\Z/|W_G(K')|$. We can determine a map $S^0 \to X$ in $\Sp^G_\omega/\cF'$ by choosing a preimage in $[\Z[0],\Z[0]]_{\Fun(BW_G(K'),\Mod_\omega(\Z))} = \Z$, this is necessarily a unit mod $|G|$. This proves the first statement.

For the second statement, note that the map $S^0\to X$ induces a map $\Phi^K S^0 \to \Phi^K X$, which after mod $|G|$ reduction induces an equivalence. Since in $\StMod(\Z W_G(K))$, we can check equivalences after mod $|G|$ reduction (every object is a retract of its mod $|G|$ reduction as $|G|$ is zero in $[\Z[0],\Z[0]]_{\StMod(\Z W_G(K))}$), this shows that the given map $S^0\to X$ induces an isomorphism on their images in $\StMod(\Z W_G(K))$, proving the second claim.

For the third claim, just note that giving a lift of $X$ to $\Pic^0(\Sp^G_\omega/\cF)$ is the same as giving an isomorphism between $\Z[0]$ and $H\overline{\Phi}^K(X)$ in $\StMod(\Z W_G(K))$, and two such lifts are equivalent if the chosen isomorphisms differ by an automorphism of $\Z[0]$ in $\Fun(BW_G(K), \Mod_\omega(\Z))$. Those automorphisms are precisely given by signs. Since we have seen above that $H\overline{\Phi}^K(X)$ is equivalent to $\Z[0]$ in $\StMod(\Z W_G(K))$, and the automorphisms of $\Z[0]$ in that category are precisely $(\Z/W_G(K))^\times$, we see that the lifts of $X$ are a free orbit of $(\Z/W_G(K))^\times/\{\pm 1\}$.
\end{proof}

\begin{corollary}
\label{cor:pic0surjective}
$\Pic^0(\Sp^G_\omega/\cF) \to \Pic^0(\Sp^G_\omega/\cF\cup \{G/K\})$ is surjective with kernel $(\Z/W_G(K))^\times/\{\pm 1\}$. In particular, the group $\Pic^0(\Sp^G_\omega/\cF)$ has order $\prod_{G/H\not\in \cF} |(\Z/W_G(K))^\times/\{\pm 1\}|$.
\end{corollary}

\begin{corollary}
\label{cor:dimfunctionindependence}
Whether an element of $\Pic(\Sp^G_\omega/\cF\cup\{G/K\})$ can be lifted to $\Pic(\Sp^G_\omega/\cF)$ depends solely on its partial generalized dimension function, i.e. image in $\cD_{\cF\cup\{G/K\}}(G)$.
More precisely, given two invertible objects $X,Y$ of $\Sp^G_\omega/\cF\cup\{G/K\}$ with the same image $d=d(X)=d(Y)$ in $\cD_{\cF\cup\{G/K\}}(G)$, and a choice of extension $\widetilde{d}\in \cD_{\cF}$ of $d$, there exists a lift of $X$ to $\Pic(\Sp^G_\omega/\cF)$ with partial generalized dimension function $\widetilde{d}$ if and only if such a lift exists of $Y$.
\end{corollary}
\begin{proof}
Under the given assumptions, the object $X\otimes Y^{-1}$ is in $\Pic^0(\Sp^G_\omega/\cF\cup\{G/K\})$, and by Corollary \ref{cor:pic0surjective} admits a lift $\widetilde{X\otimes Y^{-1}}$ to $\Pic^0(\Sp^G_\omega/\cF)$. Given a lift of $Y$ as described, we can tensor it with $\widetilde{X\otimes Y^{-1}}$ to obtain a lift of $X$ as described, and similarly vice versa.
\end{proof}

Corollary \ref{cor:dimfunctionindependence} allows us to reduce the inductive study of $\Pic(\Sp^G_\omega/\cF)$ to the easier-to-manage question of which elements of $\cD_\cF(G)$ are partial generalized dimension functions of invertible objects of $\Sp^G_\omega/\cF$.

\begin{corollary}
\label{cor:orientationbehaviour}
f $X\in \Sp^G_\omega/\cF$ is invertible and has dimension function constant $0$, the generalized dimension function is also $0$ (i.e. the orientation behaviour is also trivial). As a result, if two invertible objects have the same dimension function, they have the same generalized dimension function.
\end{corollary}
\begin{proof}
Let $G/K$ be such that $\cF\cup\{G/K\}$ is a bigger family. By inductively applying the result to the image of $X$ in $\Sp^G_\omega/\cF\cup \{G/K\}$, we may assume that image to lie in $\Pic^0(\Sp^G_\omega/\cF\cup\{G/K\})$. By Lemma \ref{lem:liftingpic0}, $H\Phi^K(X)$ is equivalent in $\StMod(\Z W_G(K))$ to $\Z[0]$ with the trivial action. So the statement follows from the observation that in $\StMod(\Z W_G(K))$, $\Z[0]$ is never equivalent to $\Z_\chi[0]$ for some nontrivial action $\chi: W_G(K)\to \{\pm 1\}$: Indeed,
\begin{align*}
[\Z[0], \Z[0]]_{\StMod(\Z W_G(K))} &= \Z/|W_G(K)|,\\
[\Z[0], \Z_\chi[0]]_{\StMod(\Z W_G(K))} &= 0,
\end{align*}
distinguishing the two, as $|W_G(K)|>1$ (otherwise, there is no nontrivial $\chi$).
\end{proof}

As a result, the dimension function determines the orientation behaviour.

We now discuss some results about $p$-complete $G$-spectra $\Sp^G_p$. On compact objects, $p$-complete spectra agree with modules over the $p$-completion $(\bS_G)^\wedge_p$, so the coefficient version of Theorem \ref{thm:pictotal1cartesian} gives $1$-cartesian diagrams
\[
\begin{tikzcd}
\cPic((\Sp^G_p)_\omega/\cF) \rar\dar & \cPic((\Sp^G_p)_\omega/\cF\cup\{G/K\})\dar\\
\cPic(\Fun(BW_G(K), \Mod_\omega(\Z_p))) \rar & \cPic(\StMod(\Z_p G)).
\end{tikzcd}
\]

Note that invertible objects of $\Fun(BW_G(K), \Mod_\omega(\Z_p))$ are given by a pair in $\Z\times \Hom(W_G(K), \Z_p^\times)$. We denote the corresponding notion of generalized dimension functions by
\[
\cD_\cF(G; \Z_p) =\prod_{G/H\not\in \cF} \Z\times \Hom(W_G(K), \Z_p^\times),
\]
and define $\Pic^0((\Sp^G_p)_\omega/\cF)$ as the kernel of
\[
\Pic((\Sp^G_p)_\omega/\cF) \to \cD_{\cF}(G;\Z_p).
\]

\begin{lemma}
$\Pic^0((\Sp^G_p)_\omega/\cF)$ is trivial, i.e. a $p$-complete invertible object with trivial partial generalized dimension function is equivalent to $S^0$. As a result, $p$-complete invertible objects are classified by their generalized dimension function.
\end{lemma}
\begin{proof}
We simply show that $\Pic^0((\Sp^G_p)_\omega/\cF) \to \Pic^0((\Sp^G_p)_\omega/\cF\cup\{G/K\})$ is injective. An object in the kernel is given by an isomorphism in 
\[
[\Z[0],\Z[0]]_{\StMod(\Z_p W_G(K))}=\Z_p/|W_G(K)|,
\]
and two such isomorphisms lead to the same object if the isomorphisms differ by an isomorphism coming from 
\[
[\Z[0],\Z[0]]_{\Fun(BW_G(K),\Mod_\omega(\Z_p))}=\Z_p.
\]
As the map $\Z_p \to (\Z_p/|W_G(K)|)^\times$ is surjective, all elements in the kernel are trivial.
\end{proof}

\begin{proposition}
\label{prop:realisabilitypcomplete}
Let $d\in \cD_\cF(G)$ be a partial generalized dimension function. Then $d$ can be realized by an invertible object of $\Sp^G_\omega/\cF$ if and only if the image of $d$ in $\cD_\cF(G;\Z_p)$ can be realized by an invertible object of $(\Sp^G_p)_\omega/\cF$ for each $p$ dividing $|G|$.
\end{proposition}
\begin{proof}
One direction is clear. For the converse, assume $d$ can be realized $p$-completely for each $p$ dividing $|G|$. We have to show that it can be realized in $\Sp^G_\omega/\cF$. Assume we have inductively already found an invertible $X\in \Sp^G_\omega/\cF\cup\{G/K\}$ whose partial generalized dimension function is the restriction of $d$. Then we need to find an isomorphism in $\StMod(\Z W_G(K))$ between $\Z[d_K]$ and $H\overline{\Phi}(X)$. By assumption we have such isomorphisms in $\StMod(\Z_p W_G(K))$ for each $p$, since $d$ can be realized $p$-completely. As
\[
\StMod(\Z W_G(K))\to \prod_{p\mid |G|} \StMod(\Z_p W_G(K))
\]
is fully faithful by Proposition \ref{prop:idempotentcompletion}, this finishes the proof.
\end{proof}

\section{Naturality}

The $1$-cartesian diagrams from Theorem \ref{thm:pictotal1cartesian} are natural in two important ways. First, observe that for a subgroup $K\subseteq G$, the geometric fixed points functor $\Phi^K: \Sp^G \to \Sp$ actually admits a refinement through \emph{genuine} $W_G(K)$-spectra, $\Phi^K: \Sp^G \to \Sp^{W_G(K)}$. For example, this comes from the fact that the $K$-fixed points functor
\[
\cS^G \to \cS^{W_G(K)}
\]
from $G$-spaces to $W_G(K)$-spaces, sends representation spheres $S^V$ to representation spheres $S^{V^K}$ and thus induces a unique symmetric monoidal colimit-preserving functor from $G$-spectra to $W_G(K)$-spectra, similar to how usual geometric fixed points are characterized by what they do on suspension spectra.

For a family of orbits $\cF\subseteq \Orb(G)$ not containing $G/K$, observe that $\Phi^K$ vanishes on all $\Sigma^\infty_+ G/H$, $H\in \cF$. Indeed, if $(G/H)^K \neq \emptyset$ for some $H$, $K$ is contained in some conjugate of $H$. But then $G/H$ cannot be contained in $\cF$, since $\cF$ is a family.

If $\cF$ contains $G/K'$ for all proper subgroups $K'\subsetneq K$, we thus get symmetric monoidal functors induced by $\Phi^K$:
\begin{align*}
\Sp^G/\cF &\to \Sp^{W_G(K)}\\
\Sp^G/(\cF\cup \{G/K\})&\to \Sp^{W_G(K)}/\{ W_G(K)/e\}
\end{align*}

These fit together to give a diagram 

\begin{equation}
\label{eq:geomfixedpointsdiag}
\begin{tikzcd}
\cPic(\Sp^G/\cF) \rar\dar & \cPic(\Sp^G/\cF\cup\{G/K\})\dar\\
\cPic(\Sp^{W_G(K)})\dar\rar & \cPic(\Sp^{W_G(K)}/\{ W_G(K)/e\})\dar\\
\cPic(\Fun(BW_G(K), \Mod_\omega(\Z)))\rar & \cPic(\StMod(\Z W_G(K)))
\end{tikzcd}
\end{equation}
in which the bottom and outer squares are $1$-cartesian, hence the upper square is also $1$-cartesian.

\begin{definition}
Let $d\in \cD_{\cF}(G)$. For every $G/K\not\in \cF$, we define an element $\varphi^K d \in \cD(W_G(K))$ as follows: To every subgroup $H\subseteq W_G(K)$, we assign its preimage $\widetilde{H}\subseteq N_G(K)$, and we get maps $W_{W_G(K)}(H)\xleftarrow{\cong} W_{N_G(K)}(\widetilde{H})\to W_{G}(\widetilde{H})$. We let 
\[
(\varphi^K d)_H \in \Z\times \Hom(W_{W_G(K)}(H),\{\pm 1\})
\]
be obtained from $d_{\widetilde{H}}\in \Z\times \Hom(W_{G}(\widetilde{H}))$ by restriction in the second factor.
\end{definition}

\begin{proposition}
\label{prop:realisabilityweyl}
Let $d\in \cD_{\cF}(G)$. Then $d$ can be realized as partial generalized dimension function of an object of $\cPic(\Sp^G_\omega/\cF)$ if and only if $\varphi^K d$ can be realized as generalized dimension function of an object of $\cPic(\Sp^{W_G(K)}_\omega)$ for each $G/K\not\in \cF$.
\end{proposition}
\begin{proof}
We proceed by induction over the size of $\cF$, starting with $\cF=\Orb(G)$, where the statement is trivial. So assume we know the statement for all families larger than $\cF$, let $G/K\not\in \cF$ with $G/K'\in \cF$ for all proper subgroups $K'\subsetneq K$. We now consider diagram \eqref{eq:geomfixedpointsdiag}.
The value of $d$ at $K$ determines an object in the lower left corner, and by induction the image of $d$ in $\cD_{\cF\cup\{G/K\}}(G)$ can be realized by an object $X$ in the upper right corner. We need to find ``gluing data'', i.e. an isomorphism between their images in the lower right corner. The image $\overline{\Phi}^K X$ in the middle right term has dimension function given by the image of $\varphi^K d$ in $\cD_{\{ W_G(K)/e\}}(W_G(K))$. Since the dimension function $\varphi^K d$ can be realized by assumption, by Corollary \ref{cor:dimfunctionindependence} $\overline{\Phi}^K X$ lifts to an object of $\cPic(\Sp^{W_G(K)})$ with the desired image in the lower left corner. Since the bottom square is $1$-cartesian, this produces the desired gluing isomorphism in the lower right corner.
\end{proof}
Proposition \ref{prop:realisabilityweyl} shows that we can inductively reduce realisability of generalized dimension functions in $\cD_{\cF}(G)$ to realisability of generalized dimension functions in $\cD(W_G(K))$. For $\cF$ a nonempty family, this involves only groups $W_G(K)$ strictly smaller than $G$. However, when analyzing realizability of entire generalized dimension functions, i.e. elements of $\cD(G)$, Proposition \ref{prop:realisabilityweyl} does not help in deciding whether we can lift from $\cPic(\Sp^G_\omega/\{ G/e\})$ to $\cPic(\Sp^G)$, since the Weyl group of $e\subseteq G$ is $G$ again.

For understanding the behaviour of the square
\[
\begin{tikzcd}
\cPic(\Sp^G_\omega) \rar\dar & \cPic(\Sp^G_\omega/\{ G/e\})\dar\\
\cPic(\Fun(BG, \Mod_\omega(\Z)))\rar & \cPic(\StMod(\Z G))
\end{tikzcd}
\]
we analyze another notion of naturality: Given a subgroup $G'\subseteq G$, we get induced symmetric-monoidal restriction functors
\begin{align*}
\Sp^G_\omega &\to \Sp^{G'}_\omega\\
\Sp^G_\omega/\{ G/e\} &\to \Sp^{G'}_\omega/\{ G'/e\}.
\end{align*}
These induce a map between the $1$-cartesian squares given by Theorem \ref{thm:pictotal1cartesian}:

\[
\begin{tikzcd}[execute at end picture={
\draw[->] (0.5,0.5) -- (0.5,-0.5);
}]
\cPic(\Sp^G_\omega) \rar\dar & \cPic(\Sp^G_\omega/\{ G/e\})\dar\\
\cPic(\Fun(BG, \Mod_\omega(\Z)))\rar & \cPic(\StMod(\Z G))\\
\\
\cPic(\Sp^{G'}_\omega) \rar\dar & \cPic(\Sp^{G'}_\omega/\{ G'/e\}) \dar\\
 \cPic(\Fun(BG', \Mod_\omega(\Z)))\rar & \cPic(\StMod(\Z G'))
\end{tikzcd}
\]

\begin{proposition}
\label{prop:gluingstableelement}
Let $d\in \cD(G)$, and assume that the image of $d$ in $\cD_{\{ G/e\}}(G)$ is realizable by an invertible object $X$ of $\Sp^G_\omega/\{ G/e\}$. Then $d$ can be realized by an invertible object of $\Sp^G$ if and only if 
for every $p$, there is an isomorphism between $\Z[d_e]$ and the image of $X$ in $\StMod(\Z P)$ (for $P$ a $p$-Sylow subgroup) which is stable in the sense of Proposition \ref{prop:stmodstable}.
\end{proposition}
\begin{proof}
By Proposition \ref{prop:realisabilitypcomplete} it suffices to realize $d$ in $\cPic(\Sp^G_p)$ for each $p$. We thus have to find an isomorphism between $H\overline{\Phi}^e(X)$ and $\Z[d_e]$ in $\StMod(\Z_p G)$. Now the statement is precisely the stable elements formula for maps in $\StMod(\Z_p G)$ from Proposition \ref{prop:stmodstable}.
\end{proof}

\section{Applications}
\label{sec:applications}
We now explicitly carry out the analysis of $\cPic(\Sp^G)$ via Theorem \ref{thm:pictotal1cartesian} for small groups $G$, leading up to $G=A_5$.
\subsection{The group $C_p$}

For $G=C_p$, Theorem \ref{thm:pictotal1cartesian} applied to $K=G/e$ yields a $1$-cartesian diagram
\[
\begin{tikzcd}
\cPic(\Sp^{C_p})\rar\dar & \cPic(\Sp_\omega/\{ G/e\})\dar\\
\cPic(\Fun(BC_p, \Mod_\omega(\Z))) \rar & \cPic(\StMod^\fin(\Z C_p))
\end{tikzcd}
\]
and Theorem \ref{thm:pullback} applied to $K=C_p$ furthermore identifies $\Sp_\omega/\{ G/e\}$ with $\Sp_\omega$, via geometric fixed points $\Phi^{C_p}$.

An invertible $C_p$-spectrum is therefore given by a dimension $d_e$ for its underlying spectrum $\Phi^e X$, together with an orientation behaviour given by an action of $C_2$ on $\Z$ if $p=2$, a dimension $d_{C_p}$ for its geometric fixed points $\Phi^{C_p} X$, and an isomorphism in $\StMod(\Z C_p)$ between $\Z[d_e]$ and $\Z[d_{C_p}]$.

For odd $p$, we have
\[
[\Z[d_e], \Z[d_{C_p}]]_{\StMod(\Z C_p)} =\begin{cases}
0 \text{ if } d_e-d_{C_p} \text{ odd,}\\
\F_p \text{ if } d_e-d_{C_p} \text{ even}.
\end{cases}
\]
So there exists an $X$ with dimensions $d_e$ and $d_{C_p}$ precisely if $n_e-n_{C_p}$ is even. Of the $\F_p^{\times}$ many isomorphisms we can choose as gluing data in that case, two lead to isomorphic $X$ precisely if they agree in $\F_p/ \{\pm 1\}$, since the automorphisms of that lift to the upper right and lower left categories of our pullback are precisely those given by signs.
This fits together with the observation established in Corollary \ref{cor:pic0surjective} that there are $\frac{p-1}{2}$ many different invertible $X$ with $\Phi^e X = S^0$ and $\Phi^{C_p} X = S^0$, and more generally that every realizable dimension function has $\frac{p-1}{2}$ different invertible $X$ realizing it.

For $p=2$, a generalized dimension function not only encodes dimensions $d_e$ and $d_{C_2}$, but also an action of $C_2$ on $\Z[d_e]$, the orientation behaviour. Here we have
\begin{align*}
&[\Z[d_e], \Z[d_{C_2}]]_{\StMod(\Z C_p)}\\
&=\begin{cases}
\F_2 \text{ if } d_e-d_{C_2} \text{ even, orientation behaviour of $d_e$ trivial}.\\
\F_2 \text{ if } d_e-d_{C_2} \text{ odd, orientation behaviour of $d_e$ nontrivial}.\\
0 \text{ if } d_e-d_{C_2} \text{ otherwise.}\\
\end{cases}
\end{align*}
It follows that in this case every combination of dimensions can be realized by a unique invertible object, and the action of $C_2$ on $\Phi^e(X)$ is determined by the parity of $n_e-n_{C_2}$.

\subsection{The group $C_p\times C_p$}

We now discuss $G=C_p\times C_p$ and illustrate how Theorem \ref{thm:pictotal1cartesian} reproduces a theorem of Borel.
The subgroup lattice of $C_p\times C_p$ consists of the trivial subgroup $e$, the full subgroup $C_p\times C_p$, and $p+1$ subgroups $C_p^i$ of order $p$.

\begin{theorem}[Borel, \cite{borel}]
Let $G=C_p\times C_p$, and let $X$ be an invertible $G$-spectrum. Then 
\[
\dim\Phi^{e}(X) - \sum_{i} \dim \Phi^{C_p^i}(X) + p\dim\Phi^{C_p\times C_p}(X)=0.
\]
\end{theorem}

We will actually show the following stronger statement:

\begin{proposition}
\label{prop:realisabilityborel}
\leavevmode
\begin{enumerate}
\item For $p$ odd, a generalized dimension function $d\in\cD(C_p\times C_p)$ can be realized precisely if all $d_{C_p^i} - d_{C_p\times C_p}$ are even, and the Borel relation holds, i.e.
\[
d_e - \sum_{i} d_{C_p^i}  + p d_{C_p\times C_p}(X)=0.
\]
\item For $p=2$, a generalized dimension function $d\in \cD(C_2\times C_2)$ can be realized precisely if the orientation behaviour at $C_2^i$ is determined by the sign of $d_{C_2^i} - d_{C_2\times C_2}$, the orientation behaviour at $e$ is given by the homomorphism $C_2\times C_2 \to \{\pm 1\}$ which is the product of $d_{C_2^i} - d_{C_2\times C_2}$ copies of the homomorphism with kernel $C_2^i$, multiplied over all $i$, and the Borel relation
\[
d_e - \sum_i d_{C_2^i} + 2 d_{C_2\times C_2}
\]
holds.
\end{enumerate}
\end{proposition}
\begin{proof}
We focus first on the case $p$ odd, and look at invertible objects in $\Sp^{C_p\times C_p}_\omega/\{ C_p\times C_p/e\}$. These have dimension functions indexed by the full group $C_p\times C_p$ and the $p+1$ many subgroups of order $p$. By Proposition \ref{prop:realisabilityweyl}, a partial dimension function $d$ can be realized if and only if, for each of the subgroups $C_p^i$, the element $\varphi^{C_p^i} d\in \cD(C_p)$ can be realized by an invertible object of $\Sp^{C_p}$. Here, $(\varphi^{C_p^i} d)_e = d_{C_p^i}$, and $(\varphi^{C_p^i} d)_{C_p} = d_{C_p\times C_p}$, and so a partial dimension function $d\in \cD_{\{C_p\times C_p/e\}}(C_p\times C_p)$ can be realized if and only if all the differences
\[
d_{C_p^i} - d_{C_p\times C_p}
\]
are even. Observe now that these partial dimension functions are spanned by the partial dimension function of the trivial $S^1$, and partial dimension functions which assign $2$ to a single $C_p^i$ and $0$ to all other groups. This partial dimension function is in fact realized by an actual representation sphere, namely the one obtained by restricting a nontrivial rotation representation $\lambda$ of $C_p$ on $\bR^2$ to $C_p\times C_p$ along a homomorphism $C_p\times C_p\to C_p$ with kernel $C_p^i$. So all realizable partial dimension functions in $\cD_{\{C_p\times C_p/e\}}(C_p\times C_p)$ can in fact be realized by invertible objects of $\Sp^{C_p\times C_p}_\omega$.

Since it is easy to see that the generating representation spheres above satisfy the Borel relation, it now suffices to check that the dimension of $\Phi^e(X)$ for an invertible object $X\in \cPic(\Sp^{C_p\times C_p}_\omega)$ is uniquely determined by the partial dimension function of the image of $X$ in $\cPic(\Sp^{C_p\times C_p}_\omega/\{ C_p\times C_p/e\})$. In light of Corollary \ref{cor:dimfunctionindependence}, this follows once we know that $\Z[n]\simeq \Z[m]$ in $\StMod(\Z C_p\times C_p)$ implies $n=m$. If this were wrong, Tate cohomology of $C_p\times C_p$ with coefficients in $\Z$ would be $(n-m)$-periodic, which is seen to be absurd by direct computation.

For $p=2$, one similarly checks that the realizable partial generalized dimension functions in $\cD_{\{C_2\times C_2/e\}}(C_2\times C_2)$ are precisely the ones where the dimensions are arbitrary, and the orientation behaviour at $C_2^i$ depends on the parity of $d_{C_2^i}-d_{C_2\times C_2}$. These are spanned by the trivial $S^1$ and restrictions of the sign representation sphere of $C_2$, all of whom satisfy the Borel relation. Similarly, one checks that these representations spheres and everything spanned by them have the given description for the orientation behaviour at $e$. Finally, nonperiodicity of Tate cohomology of $C_2\times C_2$ and Corollary \ref{cor:orientationbehaviour} finish the proof.
\end{proof}

\subsection{The group $D_{2p}$}

In the dihedral group $D_{2p}$ for odd $p$, we have four conjugacy classes of subgroups: The trivial subgroup $e$, the entire group $D_{2p}$, the normal subgroup $C_p$ and $p$ copies of $C_2$, all conjugate to each other.

\begin{proposition}
\label{prop:dihedral}
A generalized dimension function $d\in \cD(D_{2p})$ is realizable is and only if the following holds:
\begin{enumerate}
\item The orientation behaviour at $C_p$ is determined by the parity of $d_{C_p}-d_{D_{2p}}$.
\item The orientation behaviour at $e$ is determined by the parity of $d_{C_2} + d_{C_p}$.
\item We have $d_e - 2d_{C_2} - d_{C_p} + 2d_{D_{2p}} = 0$ mod $4$.
\end{enumerate}
\end{proposition}
\begin{proof}
By Proposition \ref{prop:realisabilityweyl} and the results on $C_2$, the realizable partial generalized dimension functions in $d\in \cD_{\{D_{2p}/e\}}$ are precisely those where the orientation behaviour at $C_p$ is determined by the parity of $d_{C_p}-d_{D_{2p}}$, no further constraints. In particular, the first condition is necessary.

These partial generalized dimension functions are in fact spanned by the generalized dimension function of representation spheres: The trivial $1$-dimensional representation, the $1$-dimensional sign representation $\sigma$, and the $2$-dimensional dihedral representation $\lambda$. These have generalized dimension functions as indicated in the following table. Since all Weyl groups are either trivial or have precisely two homomorphisms to $\{\pm 1\}$, we indicate the orientation behaviour by a sign subscript:
\begin{equation}
\label{eq:tabledihedral}
\begin{array}{c|c|c|c|c}
 & e & C_2 & C_p & D_{2p}\\
\hline
1 & 1_+ & 1 & 1_+ & 1\\
\sigma & 1_- & 0 & 1_- & 0\\
\lambda & 2_- & 1 & 0_+ & 0
\end{array}
\end{equation} 

In particular, the map $\cPic(\Sp^G_\omega) \to \cPic(\Sp^G_\omega/\{ G/e\})$ is surjective. We determine its kernel. For that, we have to determine all invertible objects of $\Fun(BD_{2p}, \Mod_{\omega}(\Z))$ which are equivalent to $\Z[0]$ in $\StMod(\Z D_{2p})$. Invertible objects are of the form $\Z[n]$ or $\Z_-[n]$.

One computes that Tate cohomology of $D_{2p}$ is precisely $4$-periodic. With coefficients in $\Z$, the sequence of cohomology groups reads $\Z/2p, 0, \Z/2, 0, \ldots$. With coefficients in $\Z_{-}$, the sequence reads $0, \Z/2, \Z/p, \Z/2, \ldots$. As a result, $\Z_-[n]$ is never equivalent to $\Z[0]$ in $\StMod(\Z D_{2p})$, and $\Z[n]$ is equivalent to $\Z[0]$ if and only if $n$ is divisible by $4$.

It follows that the realizable entire generalized dimension functions are generated by the dimension functions of the representation spheres from table \eqref{eq:tabledihedral} and the generalized dimension function which is $4_+$ at $e$ and $0$ everywhere else. One directly checks that the subgroup of $\cD(D_{2p})$ generated by those can be characterized as claimed.
\end{proof}

\subsection{The group $A_5$}

We now compute the possible dimension functions for $G=A_5$. The conjugacy classes of subgroups of $A_5$ are described as follows:
\[
\begin{array}{c|ccccccccc}
|H| & 1 & 2 & 3 & 4 & 5 & 6 & 10 & 12 & 60\\
\hline
H & e & C_2 & C_3 & C_2\times C_2 & C_5 & D_6 & D_{10} & A_4 & A_5\\ 
\hline
W_G(H) & A_5 & C_2 & C_2 & C_3 & C_2 & e & e & e & e\\
\end{array}
\]

By Proposition \ref{prop:realisabilityweyl} a partial generalized dimension function $d\in \cD_{\{A_5/e\}}(A_5)$ is realizable if and only if the following holds:

\begin{enumerate}
\item The number $d_{C_2\times C_2} - d_{A_4}$ is even.
\item The orientation behaviour at $C_2$ is determined by the parity of $d_{C_2}-d_{C_2\times C_2}$.
\item The orientation behaviour at $C_3$ is determined by the parity of $d_{C_3}-d_{D_6}$.
\item The orientation behaviour at $C_5$ is determined by the parity of $d_{C_5}-d_{D_{10}}$.
\end{enumerate}

We now apply Proposition \ref{prop:gluingstableelement}. Let $X$ be an invertible object of $\Sp^{A_5}_\omega/\{ A_5/e\}$. The restriction of $d$ to $\cD(C_2\times C_2)$ can be realized by an invertible object of $\Sp^{C_2\times C_2}_\omega$ if and only if 
\[
d_e- 3d_{C_2} + 2d_{C_2\times C_2} = 0.
\]
It follows that under this condition, we find an isomorphism between $\Z[d_e]$ and $H\overline{\Phi}^e X$ in $\StMod(\Z_2 C_2\times C_2)=\StMod(\Z C_2\times C_2)$. Since then
\[
[\Z[d_e],H\overline{\Phi}^e(X)]_{\StMod(\Z_2 C_2\times C_2)} \simeq [\Z[d_e],\Z[d_e]]_{\StMod(\Z_2 C_2\times C_2)} \simeq \Z/4,
\]
which just doesn't have a nontrivial $C_3$-action, the isomorphism is necessarily stable, satisfying the assumption of Proposition \ref{prop:gluingstableelement}.

For $p=3$ and $p=5$ we proceed slightly differently: Note that $A_5$ and $D_6$ have the same $3$-Sylow subgroup, and a morphism of $\StMod(\Z_3 C_3)$ is stable with respect to $\StMod(\Z_3 A_5)$ if and only if it is stable with respect to $\StMod(\Z_3 D_6)$, since $D_6$ is the normalizer of $C_3$. The analogous statement at $p=5$ holds for $D_{10}$ and $C_5$. It follows that the assumption of Proposition \ref{prop:gluingstableelement} for $p=3$ and $p=5$ can be checked after restriction to $D_6$ and $D_{10}$. Since we fully understand conditions for realizability in those cases (Proposition \ref{prop:dihedral}),
we arrive at:

\begin{proposition}
A generalized dimension function $d\in \cD(A_5)$ is realizable if and only if:
\begin{enumerate}
\item $d_{C_2\times C_2} - d_{A_4} = 0$ mod $2$.
\item The orientation behaviour at $C_2$ is determined by the parity of $d_{C_2}-d_{C_2\times C_2}$.
\item The orientation behaviour at $C_3$ is determined by the parity of $d_{C_3}-d_{D_6}$.
\item The orientation behaviour at $C_5$ is determined by the parity of $d_{C_5}-d_{D_{10}}$.
\item $d_e - 3d_{C_2} + 2d_{C_2\times C_2} = 0$.
\item $d_{C_2} + d_{C_3} = 0$ mod $2$
\item $d_{C_2} + d_{C_5} = 0$ mod $2$
\item $d_{e} - 2d_{C_2} - d_{C_3} + 2d_{D_{6}} = 0$ mod $4$
\item $d_{e} - 2d_{C_2} - d_{C_5} + 2d_{D_{10}} = 0$ mod $4$
\end{enumerate}
\end{proposition}

A somewhat arbitrary basis for the subgroup of $\cD(A_5)$ consisting of all generalized dimension functions satisfying these conditions can be computed using the Smith normal form algorithm, and is given as follows:
\[
\begin{array}{c|c|c|c|c|c|c|c|c}
e & C_2 & C_3 & C_2\times C_2 & C_5 & D_6 & D_{10} & A_4 & A_5\\
\hline
1 & 1_+ & 1_+ & 1 & 1_+ & 1 & 1 & 1 & 1\\
4 & 2_- & 2_- & 1 & 0_+ & 1 & 0 & 1 & 0\\
5 & 3_+ & 1_+ & 2 & 1_+ & 1 & 1 & 0 & 0\\
3 & 1_- & 1_- & 0 & 1_- & 0 & 0 & 0 & 0\\
0 & 0_+ & 2_- & 0 & 0_+ & 1 & 0 & 0 & 0\\
8 & 4_+ & 0_+ & 2 & 0_+ & 0 & 0 & 0 & 0\\
0 & 0_+ & 4_+ & 0 & 0_+  & 0 & 0 & 0 & 0\\
12 & 4_+ & 0_+ & 0 & 0_+ & 0 & 0 & 0 & 0\\
\end{array}
\]
The first four rows can be realized by representation spheres over $A_5$, namely the trivial and the irreducible $4$, $5$ and (either one of the) $3$-dimensional representations. The remaining four rows are more exotic, and we have no geometric construction of the corresponding invertible $A_5$-spectra. Note that there are several other interesting dimension functions contained in the span, for example the dimension function with $d_{A_5}=1$ and $d_{H}=0$ for all other subgroups.
\bibliographystyle{amsalpha}
\bibliography{bibliography.bib}

\end{document}